\documentclass[10pt, oneside]{amsart}
\usepackage{amssymb}
\usepackage{amsmath}
\usepackage{amsfonts}
\usepackage{amsthm}
\usepackage{amscd}
\usepackage{verbatim}
\usepackage{latexsym} 
\usepackage{graphicx}
\usepackage{subfig}       
\usepackage{mathrsfs}
\usepackage{epstopdf}
\usepackage{color}
\usepackage{multirow}
\usepackage{enumitem}
\setlist{leftmargin=5mm}
\usepackage{multirow}
\usepackage{hhline}
\usepackage[colorlinks=true,bookmarks=false]{hyperref}
\hypersetup{colorlinks, citecolor=blue, filecolor=black, linkcolor=red, urlcolor=green}

\renewcommand{\P}{\mathbb{P}}

\newcommand{\R}{\mathbb{R}}

\newcommand{\C}{\mathbb{C}}

\theoremstyle{definition}
\newtheorem{theorem}{Theorem}[section]
\newtheorem*{theorem*}{Theorem}
\newtheorem{lemma}[theorem]{Lemma}

\newtheorem{corollary}[theorem]{Corollary}

\newtheorem{conjecture}[theorem]{Conjecture}

\theoremstyle{remark}
\newtheorem*{remark}{Remark}
\newtheorem*{example}{Example}
\newtheorem*{illustration}{Illustration}
\newtheorem*{notation}{Notation}
\newtheorem*{definition}{Definition}

\newtheorem*{outline}{Outline of the proof}

\title[The Point of Collapse of Axis Aligned Polygons]{Glick's conjecture on the point of collapse of axis-aligned polygons under the pentagram maps}
\author{Zijian Yao}
\address{Zijian Yao, Brown University.}
\email{zijian$\_$yao@brown.edu}

\begin{document}

\maketitle

\begin{abstract} 

The pentagram map has been studied in a series of papers by Schwartz and others. Schwartz showed that an axis-aligned polygon collapses to a point under a predictable number of iterations of the pentagram map. Glick gave a different proof using cluster algebras, and conjectured that the point of collapse is always the center of mass of the axis-aligned polygon. In this paper, we answer Glick's conjecture positively, and generalize the statement to higher and lower dimensional pentagram maps. For the latter map, we define a new system -- the mirror pentagram map -- and prove a closely related result. In addition, the mirror pentagram map provides a geometric description for the lower dimensional pentagram map, defined algebraically by Gekhtman, Shapiro, Tabachnikov and Vainshtein.  \\  
\end{abstract}

\section{Introduction}\label{sec:introduction}

The pentagram map on general polygons was introduced by Richard Schwartz in \cite{Schwartz1} in 1992. It is a map $$T:\left\{\text{polygons in }\P^2\right\} \rightarrow \left\{\text{polygons in } \P^2\right\}$$ given by the following procedure: we cyclically label $n$ points (which form a polygon $P$), then draw lines $l_i$ from the $i$-th vertex to the $(i+2)$-th vertex and obtain the polygon $T(P)$ whose vertices are given by $l_i \cap l_{i+1}$, reducing modulo $ n$ whenever necessary. In other words, $T$ acts on the polygon $P$ by drawing $2$-diagonals and taking the intersections of successive diagonals as vertices of $T(P)$. Figure \ref{figure:pentagram_example} provides an example of the pentagram map on a pentagon. For convenience we will work over $\R$ throughout the paper, but our results hold over $\C$ (and other fields)  as well. We remark that the pentagram map commutes with projective transformations. Furthermore, we point our that the pentagram map is a completely integrable dynamical system, see \cite{OST} by Ovsienko, Schwartz and Tabachnikov. 

We adopt a similar labelling convention of the polygon $P$ used by Schwartz in \cite{Schwartz1}. The vertices of $P$ are labelled by indices of the same parity, while the edge between vertices $P_{i-1}$ and $P_{i+1}$ is labelled by $P_{i}$. The intersection of diagonals $P_{i-1} P_{i+3}$ and $P_{i-3} P_{i+1}$ is labelled as $Q_i$, which form the vertices of $T(P)$. (See Figure \ref{figure:pentagram_example}).
\begin{figure}
\begin{center}
\includegraphics[scale=0.55]{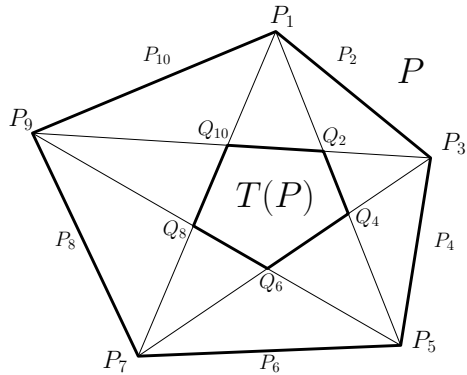} 
\end{center} 
\caption{\label{figure:pentagram_example} The pentagram map on a pentagon.}
\end{figure}

\begin{definition}
We say that a $2n$-gon $P$ is \emph{axis-aligned} if the edges $P_2, P_6, ..., $ $P_{4n-2}$ are concurrent, and the edges $P_{4}, P_8, ..., P_{4n}$ are concurrent, in other words, alternating edges go through the same point in $\P^2$.  The name comes from the fact that we can projectively transform such a polygon so that the sides are parallel to the affine $x$-axis and $y$-axis. 
\end{definition}

\begin{definition}
Let $P$ be an axis-aligned polygon. We define the center of mass $\mathscr C(P)$ of $P$ as follows: first we apply a projective transformation $\varphi$ to $P$ such that $\varphi (P)$ has edges parallel to $x$-axis and $y$-axis alternatively. Note that $\varphi (P)$ sits in the affine plane $\R^2 \subset \P ^2$, with standard coordinates. Let $C$ be the centroid of $\varphi (P)$, namely, $$C = \displaystyle \frac{1}{2n}  \left(\varphi (P_1) + \varphi (P_3) + ... + \varphi (P_{4n-1})\right)$$ is the point obtained by taking the average of the coordinates of all vertices in $\varphi (P)$. The center of mass of $P$ is defined as $$\mathscr C (P) = \varphi^{-1} (C) = \varphi^{-1} \left( \frac{1}{2n} \left(\sum_{j = 1}^{2n}  \varphi(P_{2j-1})\right) \right). $$ 
\end{definition}

\begin{remark} The definition of $\mathscr C (P)$ is independent of different choices of projective transformation $\varphi$, such that the edges of $\varphi (P)$ are parallel to $x$-axis and $y$-axis respectively. In other words, $\mathscr C (P)$ is well defined. 
\end{remark}

\begin{remark}
If the polygon $P$ has edges already parallel to the $x$-axis and $y$-axis, then we may put a unit mass at each vertex of the polygon, then $\mathscr C(P)$ will be the physical center of mass of $P$.  This justifies the choice of terminology ``center of mass'' in the paper.
\end{remark}

Axis-aligned polygons are of particular interest since Schwartz proved the following surprising result: 

\begin{theorem}[Schwartz $'01$ \cite{Schwartz2}, Schwartz $'08$ \cite{Schwartz3}, Glick $'11$ \cite{Glick2}]{\label{T001}} Let $S$ be the set of all axis-aligned $2n$-gons. Then there exists a generic subset $S' \subset S$ such that for any axis-aligned $A \in S'$, $T^{n-2}(A)$ has all its vertices lying on two lines, with each vertex alternating. Consequently, all vertices of $T^{n-1}(A)$ collapse to a single point. 
\end{theorem}

\begin{remark}
The proof of Theorem \ref{T001} consists of a proof of some particular lower bound and upper bound. Both Schwartz (in \cite{Schwartz3}) and Glick (in \cite{Glick2}) explicitly showed the upper bound (using different methods). Schwartz's proof is easily adapted to prove the lower bound, even though he did not explicitly mention it. In \cite{Schwartz2}, Schwartz proved a generalization of the collapsing result for higher diagonal maps. His argument also works for $2$-diagonal maps, in which case the statement of Theorem \ref{T001} is recovered, though again, this point is not emphasized in the paper. 
\end{remark}

\begin{example} Figure \ref{figure:collapse_example} gives an example of the theorem for $n =3$. \end{example}

\begin{definition}
For an axis aligned polygon $A$ that satisfies the conditions in Theorem \ref{T001}, we call the point $p$, where all vertices of $T^{n-1}(A)$ collapse to, \emph{the point of collapse} of $A$ under the pentagram map. 
\end{definition}

\begin{figure}
\begin{center}
\includegraphics[scale=0.6]{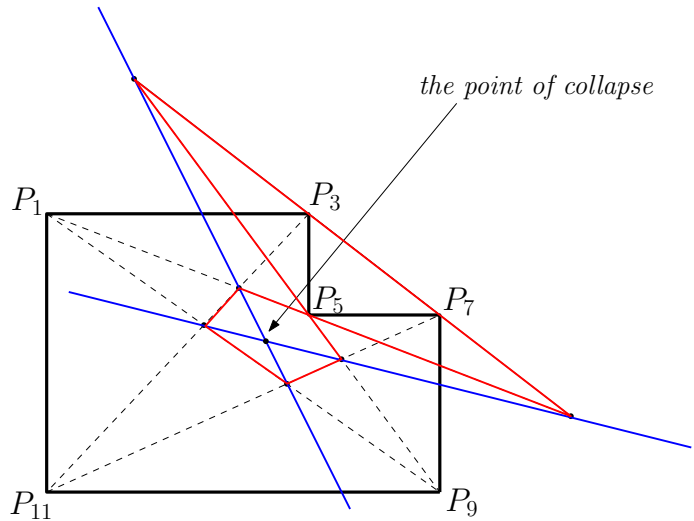} 
\end{center} 
\caption{\label{figure:collapse_example} The point of collapse of an axis-aligned 6-gon under $T^2$.}
\end{figure}

\begin{remark}
We mention that $T^{-1}(A)$ gives the line at $\infty,$ so there is a nice duality  between the line at $\infty$ and the point of collapse for axis-aligned polygons. 
\end{remark}

Glick observed an interesting and surprising pattern through computer experimentation, and has the following conjecture \cite{Glick3}: 

\begin{conjecture}[Glick, unpublished]
Let $P \subset \P^2$ be an axis-aligned 2n-gon, let $p = T^{n-1} (P)$ be its point of collapse under the pentagram map, then $$p = T^{n-1} (P) = \mathscr C (P).$$Namely, the point of collapse is equal to the center of mass.
\end{conjecture}

\begin{remark}
Since $\mathscr C (P)$ is projectively invariant, this is equivalent to the following: suppose that the edges of $P$ are parallel to $x$- and $y$-axis respectively, where $P$ has vertices $(x_1,y_1), (x_2, y_2), \ldots,(x_{2n},y_{2n})$, where $x_{2k-1} = x_{2k}, y_{2k} = y_{2k+1}$,  then the coordinate of the point of collapse $p$ coincide with the center of mass $$ p = T^{n-1} (P) = (\frac{\sum_{i}^{}x_i}{2n},\frac{\sum_{i}^{}y_i}{2n}).$$ 
\end{remark}

\begin{notation} From now on in the paper, axis-aligned polygon will only refer to the polygon with edges parallel to the $x$-axis and $y$-axis respectively. In this case, we can restrict ourselves to $\R^2 \subset \P^2$ (and later on when we discuss higher pentagram maps, $\R^m \subset \P^m$).\\
\end{notation}

In Section \ref{sec:pentagram} of this paper, we answer the conjecture positively: 

\begin{theorem}{\label{T002}} Let $\mathcal P = P_{1} P_{3} .... P_{4n-1} \subset \R^2$  be a generic axis-aligned polygon, let $T$ be the pentagram map, then $T^{n-1} (\mathcal P)$ collapses to the center of mass of $\mathcal P$, in other words, $$T^{n-1}(\mathcal P) = \mathscr C (\mathcal P) = \frac{1}{2m}(P_1 +P_3... + P_{4n-1}).$$ 
\end{theorem}

The key idea in the proof is based on a refinement of certain liftings that Schwartz constructed in \cite{Schwartz2} to prove the generalization of Theorem \ref{T001} that we mentioned, which is in fact really a generalization of Desargue's theorem.  We will define similar liftings of a sequence of $n$-gons into $\R^n$ in a controlled manner. The crucial point is that after these controlled liftings of the specifically chosen polygons into $\R^n$, the centroids of the lifted polygons coincide at a single point. 
\vspace*{0.1in}

In the remaining sections we generalize Theorem \ref{T002}  to higher and lower dimensional pentagram maps. We mention that in \cite{Glick1}, Glick names such collapsing phenomenons the Devron property, and proves that a higher and a lower dimensional analog of the pentagram map \cite{GSTV} both have the Devron property. Our generalization strengthens Glick's result in the following sense:  we give the specific number of iterations needed to reach the point of collapse; we also prove that the point of collapse of  an ``axis-aligned polygon'' in dimension $k$, where $k = 1$ or $k \ge 3$, equals the ``center of mass'' in some appropriate sense. 

More precisely, in Section \ref{sec:higherpentagram} we show the following theorem regarding axis-aligned polygons in $\R^m$ (we defer the precise definition of the higher pentagram map $T_m$ and axis-aligned polygons in $\R^m$ until Section \ref{sec:higherpentagram} ). 

\begin{theorem}{\label{T003}}  Let $P$ be a generic axis-aligned $mn$-gon in $\R^m$, and $T_m$ the $m$-dimensional corrugated pentagram map,  then $T_m^{n-1} (P)$ collapses to a point. Furthermore, it collapses to the center of mass of $P$, in other words, $$T_m^{n-1}(P) = \mathscr C (P).$$ 
\end{theorem}

In Section \ref{sec:lowerpentagram} of the paper, we discuss the lower pentagram map.  While the higher dimensional pentagram map is defined using a similar geometric construction in $\P^m$, the lower dimensional analog is a certain system defined algebraically on points in $\P^1$, by Gekhtman, Shapiro, Tabachnikov and Vainshtein in \cite{GSTV}.

In this paper, however, we use a geometric construction to reinterpret the lower pentagram map, by defining the mirror-pentagram (MP) map. This provides a way to construct  $1$-diagonals in some appropriate sense, and to some extent, gives the correct geometric definition of the pentagram map in dimension 1.  We prove a similar result to Theorem \ref{T002} for the MP map. Consequently,  we get the following interesting result (Theorem \ref{T005}) as a corollary: \\

Fix an arbitrary positive integer $n$, consider two sequences of cyclically ordered  $n$ points in $\P^1$: $$A_0 = \left\{ X_{(0, 2)}, X_{(0, 4)}, X_{(0, 6)}, ... , X_{(0, 2n)} \right\} $$ and $$A_1 = \left\{ X_{(1, 1)}, X_{(1, 3)}, X_{(1, 5)}, ... , X_{(1, 2n-1)} \right\} .$$  We first form two infinite sequences of numbers by repeating $A_0$ and $A_1$ repetitively (cyclically) and arrange them in the following patterns in the plane, then inductively generate (possibly) infinite collections of $n$ repeated points:

\begin{equation*}
\arraycolsep=0.3pt
\medmuskip = 1mu
\begin{array}{ccccccccccccccccccccccc}
...  \hspace*{0.15cm}&& X_{(0, 2n)}  & & X_{(0,2)}  & & X_{(0,4)}  & & ...  & &  X_{(0,2n)} & & X_{(0, 2)} & &  \\
& && X_{(1, 1)} & & X_{(1, 3)} & & ... & &... && X_{(1, 1)} && ...  \\
...   \hspace*{0.15cm}& & X_{(2, 2n)}  & & X_{(2,2)}  & & X_{(2,4)}  & & ...  & &  X_{(2,2n)} & & X_{(2, 2)} & &  \\
& & && && \vdots  \\
& && X_{(2m-1, 1)} & & X_{(2m-1, 3)} & & ... & &... && X_{(2m-1, 1)} && ...  \\
...  \hspace*{0.15cm}& & X_{(2m, 2n)}  & & X_{(2m,2)}  & & X_{(2m,4)}  & & ...  & &  X_{(2m,2n)} & & X_{(2m, 2)} & &  \\
& & && && \vdots  
\end{array}
\end{equation*}\\ 
where the row $A_{i+1}$ is determined by row $A_{i-1}$ and $A_{i}$ by the following rules: for each diamond pattern:
$$ \arraycolsep=1pt
\begin{array}{cccccc}
&  X_{(i-1, k)} \\
X_{(i, k-1)} && X_{(i, k+1)} \\
& X_{(i+1, k)}
\end{array}, $$
we require the cross ratio $$[X_{(i-1, k)}, X_{(i, k-1)}, X_{(i+1, k)},X_{(i, k+1)}] = -1,$$ where the cross ratio of $a, b, c, d \in \P^1$ is defined to be $$[a, b, c, d] = \frac{(a - b) (c-d)}{(b-c) (d-a)}.$$ Such systems are considered, for example, in \cite{ABS} by Adler, Bobenko and Suris. 

\begin{remark} Our system starts with a constant row and ends with another constant row, which is similar in spirit to the celebrated frieze pattern discovered by Conway and Coxeter in \cite{CC}. For this reason, we denote our system by \emph{cross ratio frieze pattern}, in honour of Conway and Coxeter. 
\end{remark}

Now we can state the corollary we prove in Section \ref{sec:lowerpentagram}:

\begin{theorem}{\label{T005}} Let $A_0$ and $A_1$ be as defined, which are the two starting rows in the cross ratio frieze pattern. Suppose that $$X_{(0,2)}  =  X_{(0,4)} = ... = X_{(0,2n)}  = \infty \in \P^1,$$ then $$X_{(2n-1,k)} = X_{(2n, k+1)} = \frac{1}{n} \Big(X_{(1,1)} + X_{(1,3)} + ... +  X_{(1, 2n-1)}\Big) $$ for all $k =1, 3, ..., 2n-1$. 
\end{theorem}

\begin{example}
We end the introduction with an example illustrating Theorem \ref{T005},   where $n = 3$ and $A_1 =\{7, 5, -3\}$:  \\
\begin{center}
\includegraphics[scale=0.75]{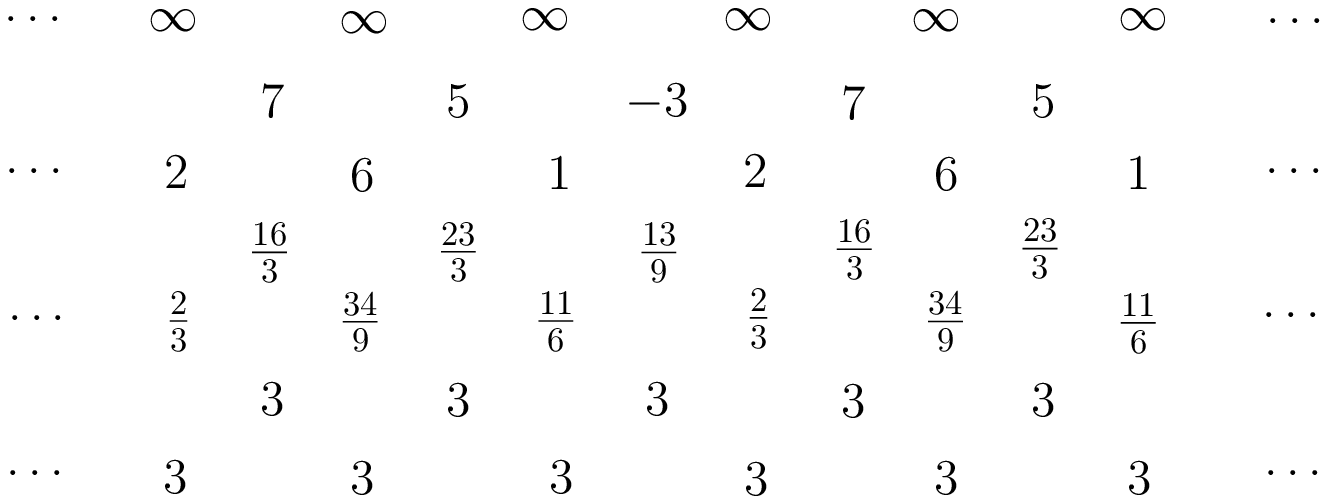} 
\end{center}  
\vspace*{0.1cm}
As the theorem predicted, both $A_5$ and $A_6$ are constant rows consisting of $$3 = \frac{1}{3} (7+5-3).$$ 
\end{example}
\vspace*{0.5cm}


\section{{The point of collapse of the pentagram map}} \label{sec:pentagram}

In this section, we prove Theorem \ref{T002} which states that the point of collapse of an axis-aligned polygon equals its center of mass. As remarked in the introduction, the proof of Theorem \ref{T002} relies on a refinement of a construction in \cite{Schwartz2} by Schwartz, where he considered a certain collection of $np$ lines that go through $p$ points for $p \ge 3$. He showed that under a predictable iterations of the $p$-diagonal map, these $np$ lines turn into a collection of $np$ points that sit on $p$ lines. The proof involves two parts, the first part (upper bound) says that the phenomenon above has to happen before a certain number of iterations, and the other part says that for each $p \ge 3$, there exists a constructible example where the upper bound is reached.  


For the upper bound, Schwartz considered a finite collection of sets of $k$ points, and lifted each set of $k$ points into general positions in $\R^k$, so the $k$ lifted points in $\R^k$ span a hyperplane. He then considered the hyperplanes spanned by the lifted sets of points. It is easy to show that generically any $d$ hyperplanes intersect transversally to form a $d$ flat (a codimension $d$ plane) in $\R^k$, and that the $d$ flat again intersects with certain faces of polytopes in space transversely to form finitely many points. He proveded the collapsing theorem by projecting these points back to $\R^2$ and realizing that the $p$-diagonal map is giving by taking such intersections of hyperplanes. 

We point out that the liftings defined in \cite{Schwartz2} do not imply our theorem, since the liftings we require are more constrained. More specifically, we have the additional requirements on the collection of points that some points degenerate in $\R^2$, and on the liftings into $\R^n$ such that certain points have the same $x_3, x_4, ..., x_n$ coordinates. The latter assumption is crucial for us to keep track of the center of mass for certain polytopes in $\R^n$. Under these assumptions, it is not immediately clear that the hyperplanes spanned by certain lifted points are in general position.  
  
\begin{outline}[Theorem \ref{T002} ]
First we form $n-1$ sequences $A_1, A_3, ..., A_{2n-3}$ of $n$ points in $\R^2$ from the $2n$ vertices of a given axis-aligned polygon $P$. We use a controlled lifting (the \emph{parallel lifting}) to lift each $A_j$ into $\widetilde A_j \subset \R^n$. Lemma \ref{lemma:parallel_lifting_of_centroids} asserts that for all $j$, the center of masses of the lifted $n$ points in $A_j$ coincide.  The $n$ points in $\widetilde A_j $ span a hyperplane in $\R^n$, and Lemma \ref{lemma:lifting_into_general_positions} says that any subset of the hyperplanes intersect transversely. The parallel lifting of all the $(n-1)n$ points that we started with also form certain \emph{prisms} in $\R^n$. We take intersections of several hyperplanes, and then intersect with certain faces (\emph{cyclic skeletons}) of the prisms to get $n$ points in space. Lemma \ref{lemma:perfect_lifting_existence} guarantees that what we get from the intersection process are always $n$ points. Then we project the $n$ points back to $\R^2$. Lemma  \ref{lemma:intersection_gives_mating_process} and \ref{lemma:fully_sliced_mating_process} shows that iterations of the pentagram map on the polygon $P$ can be obtained by projecting these points. In the end we carefully keep track of where the center of mass of the polygon gets lifted and projected into, and the theorem follows from the fact that in $\R^n$, the lifted center of mass ``gets stuck'' on the line of intersection of all the $n-1$ hyperplanes. 
\end{outline}

\subsection{Proof of Theorem \ref{T002}} \label{subsection:2.1}

We proceed to the proof by first defining several terms, the first three of which are given by Schwartz in \cite{Schwartz2}.

\begin{definition}
A collection of $n$ points in $\R^2$ in general position is called an \emph{$n$-point.} In general, for $2 \le m \le n$, an $n$-point in $\R^m$ is defined similarly. 
\end{definition}

\begin{definition} A sequence of $n$ points in $\R^n$ in general position is called an $n$-\emph{joint}. A sequence of $n$ parallel lines in $\R^n$  in general position is called a \emph{prism}. 
\end{definition}

\begin{definition}
Let $J_1, J_3$ be 2 $n$-joints, then write $J_1J_3$ as the collection of $n$ lines formed by corresponding pairs of points from $J_1$ and $J_3$. A sequence of m $n$-joints $J_1, J_3 ..., J_{2m-1}$ is called an m-\emph{polyjoint} if each $J_{2k-1} J_{2k+1}$ is a prism. 
\end{definition}

Now consider an axis-aligned polygon $P = P_1 P_3 .... P_{4n-1} \in \R^2 \subset \R^n$, where we embed $\R^2$ in $\R^n$ by adding $0$ to $x_3, x_4, ..., x_n$ coordinates. Consider the following $(n-1)$ sequences of $n$-points (where elements are given in their linear order): 
\vspace*{0.1cm}
\begin{eqnarray*} & A_1 &= \left\{P_1, P_5, ..., P_{4n -7}, P_{4n -3}\right\}  \\&  A_3 & = \left\{P_3, P_7, ...,P_{4n -5}, P_{4n-1} \right\} \\ & A_5 &= \left\{P_5, P_9, ..., P_{4n-3}, P_1\right\} \\  && \quad  \ldots \\ & A_{2n-3}  &= \left\{P_{2n-3}, P_{2n+1}, ...,  P_{2n -7} \right\} \qquad  
\end{eqnarray*}

Note that as sets, $A_1 = A_5 = ...$ , but as linearly ordered sequences, they are pairwise distinct. Let $A $ be the collection $A = \left\{ A_1,  A_3, ... , A_{2n-1} \right\}$, so $A$ has $(n-1) $ elements and each element is one of the $A_{j}$, hence $A$ contains $(n-1)n$ points regarded as a multi-set of points. We want to lift all points in $A$ simultaneously into $\R^n$. In order to differentiate points $P_k$ in different ordered set $A_j$, we label a point $P_k$ by $P_k(j)$ if $P_k \in A_j$. Hence we have   
\begin{eqnarray*} & A_1  &= \left\{P_1(1),\quad  P_5(1), \quad ..., \quad P_{4n-7}(1), \quad P_{4n-3}(1)\right\}   \\& A_3 & = \left\{P_3(3), \quad P_7(3), \quad ..., \quad P_{4n -5}(3), \quad P_{4n-1}(3) \right\} \\ &A_5  &= \left\{P_5(5), \quad P_9(5), \quad ..., \quad P_{4n-3}(5), \quad P_1(5)\right\} \\  && \quad \ldots \\ & A_{2n-3}  & = \left\{P_{2n-3}(2n-3),  P_{2n+1}(2n-3),  ...,  P_{2n -7} (2n-3)\right\} \quad 
\end{eqnarray*}  
and we can think of $P_k (j), P_k (j+4), ...$ as different points but coincide in $\R^2 \subset \R^n$.  Note that all points in $A$ have coordinates $(x_1, x_2, 0, ..., 0)$ for some $x_1, x_2$. \\

We define a \emph{lifting} of $A$ to be a way to lift each $A_k$ into a joint in $\R^n$ while fixing the $x_1, x_2$ -coordinates for every point. In other words, let $\pi: \R^n \rightarrow \R^2$ be the canonical projection sending $(x_1, x_2, x_3, ..., x_n) \mapsto (x_1, x_2, 0, ..., 0)$, then a lifting of the points $P_k(J)$ results in points $\widetilde P_k(j)$ such that $\pi (\widetilde P_k(j)) = P_k(j)$.

\begin{definition}
A \emph{parallel lifting} $L$ of the sequence of $n$-points $A_1, A_3, ..., A_{2n-3}$ is a family of liftings defined as follows: $L$ sends $A_1, ..., A_{2n-3}$ to $\widetilde A_1, ..., \widetilde A_{2n-3}$, where the $l^{th}$ point of each $\widetilde A_{k}$ has the same $x_3, x_4, ..., x_n$ coordinates, for all $l =1, 2, ..., n$. This is to say that, we take the points that occupy the same position in each $A_k$, and lift them to the same \emph{height}, by which we mean the same $x_3$ to $x_n$ coordinates. 
\end{definition}

For example, suppose that in $\R^2$ the points $P_{2j-1}$ have coordinates $(\alpha_{2j-1}, \beta_{2j-1})$, then parallel lifting requires (on the first element of each $A_k$) that, in Cartesian coordinates: 
\begin{align*}
\widetilde P_{1}(1) & = (\alpha_1, \beta_1, c_3, c_4, ..., c_n), \, \\ \widetilde P_3(3) & =  (\alpha_3, \beta_3, c_3, c_4, ..., c_n), \\ & \ldots \\ \widetilde P_{2n-3}(2n-3) & = (\alpha_{2n-3}, \beta_{2n-3}, c_3, c_4, ..., c_n) 
\end{align*}
where $c_3, ..., c_n \in \R^n$. Similarly, the points $\widetilde P_5(1), \widetilde P_{7}(3), ..., \widetilde P_{2n+1} (2n-3)$ will have the same \emph{height}, and so forth. 


\begin{lemma}{\label{lemma:parallel_lifting_into_joint}} A parallel lifting lifts the sequence of $n$-points defined above $A_1, A_3, ..., $ $ A_{2n-3}$ to a polyjoint $(\widetilde A_1, \widetilde A_3, ..., \widetilde A_{2n-3} ).$
\end{lemma}

\begin{proof} Let the $j^{th}$ element of $A_{k}$ be $A_k(j)$ for $1 \le j \le n$, and similarly define $\widetilde A_k (j)$. Hence $A_k(j) = P_{k+4(j-1)}(k)$ by definition of notations. By construction, for each odd $k, 1 \le k \le 2n-1$, we know that the $n$ lines $$A_k(1)A_{k+2}(1), A_k(2)A_{k+2}(2), ... , A_k(n)A_{k+2}(n)$$ are parallel, to either $x_1$-axis or $x_2$-axis in $\R^n$ depending on $k$. The parallel lifting is defined such that the line through the corresponding pairs $\widetilde A_k (j)$ and $\widetilde A_{k+2}(j)$ is parallel to $A_k(j) A_{k+2} (j)$, hence each $\widetilde A_k \widetilde A_{k+2}$ is a prism in $\R^n$. 
\end{proof}


Our next lemma explains why a parallel lifting is required.

\begin{lemma}{\label{lemma:parallel_lifting_of_centroids}} Suppose that a polyjoint $(\widetilde A_1, \widetilde A_3, ..., \widetilde A_{2n-3} )$ is lifted from $A_1, A_3, ..., $ $A_{2n-3}$ by a parallel lifting. Let $ C_i \in \R^n$ be the centroid of the $n$-vertices in $\widetilde A_i$ for each $i = 1, 3, ..., 2n -3$.  Then
\begin{equation*}  C_1 = C_3 = ... = C_{2n-3}; \quad \text{and }
 \pi (C_1) = ... = \pi (C_{2n-3}) = \mathscr C (P)
\end{equation*}
\end{lemma}
\begin{proof}
Note that $\pi (C_i)$ is precisely the centroid of $A_i$. Since $P$ is axis-aligned, we know that $\pi(C_1) = \pi (C_3)$. Since $A_1 = A_5 = ...$ and $A_3 = A_7 = ...$ as sets, it is clear that the second part of the lemma holds. 

For the first part of the lemma we only need to show that each $C_i$ has the same $x_j$ coordinates for $j \ge 3$, which follows immediately from the definition of the parallel lifting.
\end{proof}

Let $J$ be an $n$-joint in $\R^n$, we denote by $|J|$ the subspace in $\R^n$ spanned by points in $J$. If points in $J$ are in general position, then $|J|$ is a hyperplane in $\R^n$.


\begin{lemma}{\label{lemma:lifting_into_general_positions}} There exists an axis-aligned polygon with a parallel lifting of $A_1, ...,$ $A_{2n-3}$ such that all subspaces $|\widetilde A_1|, |\widetilde A_3|, ..., |\widetilde A_{2n-3}|$ are hyperplanes in general position. 
\end{lemma}
\begin{proof} 

There is a geometric proof and an algebraic one. The geometric one is more intuitive, but for simplicity's sake, we present the algebraic proof.  Suppose $n = 2l +1$, the case where $n$ is even is analogous. 

We construct a lifting $L_0: (\R^2)^{n(n-1)} \rightarrow (\R^{n})^{n(n-1)}$ as follows: 

First we give coordinates to the points $P_k$ in $\R^2$. If we assume that $P_1 P_3$ is parallel to the $x$-axis, then the points $P_3, P_7, ..., P_ {4n-1}$ are determined (uniquely) by $P_1, P_5, ..., P_{4n-3}$. Let the coordinates of $P_1,  ..., P_{4n-3}$ in $\R^2$ be 
$$P_1 = (a_1, b_1), \quad P_5 = (a_2, b_2), \quad  ..., \quad P_{4n-3} = (a_n, b_n),$$
then the other points have coordinates
$$P_3= (a_2, b_1),\quad P_7 = (a_3, b_2), \quad  ..., \quad P_{4n-1} = (a_1, b_n).$$

For the lifting $L_0$, we make a specific choice so that for each $A_k$ ($k$ being odd), $\widetilde A_k(1) = A_k(1)$ and $\widetilde A_k(2) = A_k(2)$, while for $j \ge 3$, $A_k(j)$ is lifted to $\R^n$ so that the $x_j$ coordinates are $1$ and the $x_i$ coordinate is $0$ for all $i\ge 3$, $i \ne j$. in other words, for each $A_k$, we fix the first two points in $\R^2$, and lift the $j^{th}$ point up by $1$ in $x_j$ direction only. Without loss of generality, we may also assume $a_1 = b_1 = 0$, so we choose $P_1$ to be the origin in $\R^n$. 

Next we will find the normal vector $v_k$ to each hyperplane $|\widetilde A_k|$, and show that there exists a polygon $P$ such that under the lifting $L_0$, the vectors $v_1, ..., v_{2n-3}$ are linearly independent in $\R^n$. This shows that $|\widetilde A_1|, |\widetilde A_3|, ..., |\widetilde A_{2n-3}|$ are in general position. We calculate each $v_k$ explicitly, which is the vector spanning the orthogonal complement of the subspace formed by the $(n-1)$ vectors $$\widetilde A_k(2) -\widetilde A_k(1), \quad  \widetilde A_k(3) -\widetilde A_k(1),\quad  ...,\quad  \widetilde A_k(n) -\widetilde A_k(1).$$

Let $e_1, e_2, ..., e_{n}$ be the standard basis in $\R^n$, and let $e = (e_1, ..., e_n)$ regarded as a vector with entries being the basis, then it is clear that 
$$v_k = \det (e, \quad \widetilde A_k(2) -\widetilde A_k(1), \quad  ...,\quad  \widetilde A_k(n) -\widetilde A_k(1) )^T$$

As an example we calculate $v_1$ and $v_3$, where $a_1 = b_1 = 0$. \\

\begin{equation} 
\nonumber v_1 = \det \begin{bmatrix}
e_1 & e_2 & e_3 & \ldots & e_n \\ 
a_2 & b_2 & 0  & \ldots & 0  \\
a_3 & b_3 & 1  & \ldots & 0  \\
\vdots  & & &  \ddots \\
a_n & b_n & 0  & \ldots & 1  \\
\end{bmatrix}  \hspace*{0.5in}
\nonumber v_3 = \det \begin{bmatrix}
e_1 & e_2 & e_3 & \ldots & e_n \\ 
a_3 - a_2 & b_2 & 0  & \ldots & 0  \\
a_4 - a_2 & b_3 & 1  & \ldots & 0  \\
\vdots  & & &  \ddots \\
a_n - a_2  & b_n & 0  & \ldots & 1  \\
\end{bmatrix} \\
\end{equation} \\

For this particular lifting $L_0$, the determinants are easy to compute, since columns $3$ to $n$ of the matrix remain the same for each $k$, which is mostly $0$ except for a sub-diagonal of $1$ and the top row being part of the basis. Hence we have the $v_k$:

\begin{align*}
v_1 & = (b_2, -a_2, \det \begin{bmatrix} a_2 & b_2 \\ a_3 & b_3   \end{bmatrix} , \ldots, \det \begin{bmatrix} a_2 & b_2 \\ a_n & b_n   \end{bmatrix}  )   \\
v_3 & = (b_2, a_2 - a_3, \det \begin{bmatrix} a_3 - a_2 & b_2 \\ a_4 - a_2 & b_3   \end{bmatrix} , \ldots, \det \begin{bmatrix} a_3 - a_2 & b_2 \\ a_1 - a_2  & b_n   \end{bmatrix}  )  \\
& \ldots \\
v_{2n-3} & =  v_{4l-1} = (b_{l+1}-b_l, a_{l+1} - a_{l+2}, \det \begin{bmatrix} a_{l+2} - a_{l+1} & b_{l+1} - b_{l} \\ a_{l+3} - a_{l+1} & b_{l+2} - b_l  \end{bmatrix} ,  \ldots, \\  & \hspace*{2.02in} \det \begin{bmatrix} a_{l+2} - a_{l+1} & b_{l+1} - b_{l} \\ a_{l} - a_{l+1} & b_{l-1} - b_l  \end{bmatrix}) 
\end{align*}

Now we set $a_j = j -1$ for all $j = 1, ..., n$. (Note that this agrees with the fact that $a_1 =0$.) Under this assumption, we calculate that

\begin{align*}
          v_1 & = (b_2, \:\: -1, \:\:  b_3 - 2 b_2, \:\:  b_4 - 3 b_2, \:\:  ...,\:\:  b_{n-1} - (n-2)b_2, \:\:  b_n - (n-1)b_2 )   \\
          v_3 & = (b_2, \:\:  -1, \:\:  b_3 - 2 b_2, \:\:  b_4 - 3 b_2, \:\:  ..., \:\: b_{n-1} - (n-2)b_2,\:\:   b_n + b_2   ) \\ 
     \begin{split}
         v_5  & = (b_3 -b_2, \:\:  -1,  \:\:  (b_4 - b_2) - 2 (b_3 - b_2),   \:\:  (b_5 - b_2) - 3 (b_4 - b_2) ,  \:\:   ..., \\ & \hspace*{1.6in}   (b_n - b_2) - (n-2) (b_3 - b_2), \:\:   (b_1 - b_2) + (b_3 - b_2)   )   
     \end{split}  \\
     \begin{split}
         v_7  & = (b_3 -b_2, \:\:  -1,  \:\:  (b_4 - b_2) - 2 (b_3 - b_2),   \:\:  (b_5 - b_2) - 3 (b_4 - b_2) ,  \:\:   ..., \\ & \hspace*{1.95in}   (b_n - b_2) +2 (b_3 - b_2), \:\:   (b_1 - b_2) + (b_3 - b_2)   )   
     \end{split}  \\
     \begin{split}
         v_9  & = (b_4 -b_3, \:\:  -1,  \:\:  (b_5 - b_3) - 2 (b_4 - b_3),   \:\:  (b_6 - b_3) - 3 (b_4 - b_3) ,  \:\:   ..., \\ & \hspace*{1.42in}   (b_n - b_3) -(n-3) (b_4 - b_3), ... , \:\:   (b_2 - b_3) + (b_4 - b_3)   )   
     \end{split}  \\
     \begin{split}
         v_{11}  & = (b_4 -b_3, \:\:  -1,  \:\:  (b_5 - b_3) - 2 (b_4 - b_3),   \:\:  (b_6 - b_3) - 3 (b_4 - b_3) ,  \:\:   ..., \\ & \hspace*{1.78in}   (b_n - b_3)  + 3 (b_4 - b_3), ... , \:\:   (b_2 - b_3) + (b_4 - b_3)   )   
     \end{split}  \\
                & \ldots  \\
     \begin{split}
        v_{4l-3} & = (b_{l+1} - b_l, \:\:  -1, \:\:  (b_{l+2}-b_l)-2(b_{l+1} - b_l), \:\:   (b_{l+3}-b_l)-3(b_{l+1} - b_l), \:\:  ...,  \\ & \hspace*{0.85in}   (b_{n}-b_l)-(l+1)(b_{l+1} - b_l), \:\:   (b_1-b_l)+ (l-1) (b_{l+1} - b_l), \:\: ...)
     \end{split}  \\
     \begin{split}
        v_{4l-1} & = (b_{l+1} - b_l, \:\:  -1, \:\:  (b_{l+2}-b_l)-2(b_{l+1} - b_l), \:\:   (b_{l+3}-b_l)-3(b_{l+1} - b_l), \:\:  ...,  \\ & \hspace*{1.2in}   (b_{n}-b_l)+ l (b_{l+1} - b_l), \:\:   (b_1-b_l)+ (l-1) (b_{l+1} - b_l), \:\: ...)
     \end{split}  
\end{align*}

It suffices to show that $[v_1, v_3, ..., v_{4l-1}]^T$ has full rank $n-1$, or equivalently, we can show that  $[v_1, v_3, ..., v_{4l-1}, v_0]^T$ has non-zero determinant for some vector $v_0$. Let $u_{4i-3} = v_{4i-1} - v_{4i-3}$ for $i = 1, 2, ..., l$. Then it suffices to show that the matrix $$M =[u_1, u_5, ..., u_{4l-3}, v_3, v_7, ..., v_{4l-1}, v_0]^T$$ has non-zero determinant. Note that the matrix $[u_1, u_5, ..., u_{4l-3}]^T$ has the following form:
\begin{equation*}
\begin{bmatrix}
0 & 0 & \ldots &0 &   0 &  \ldots & 0 & 0 & n b_2\\
0 & 0 &  \ldots & 0 & 0 & \ldots &0 & n(b_3 - b_2) &   0\\
0 & 0 &  \ldots & 0 & 0 & \ldots & n(b_4 - b_3) & 0 &  0\\
\vdots & \vdots & & \vdots && \ldots  \\ 
0 & 0 &  \ldots &0 & n(b_{l+1} - b_l) & \ldots & 0 & 0 & 0 \\
\end{bmatrix}
\end{equation*}

If we write $M = \begin{bmatrix} M_1 & M_2 \\ M_3 & M_4 \end{bmatrix}$ where $M_1$ is $(l+1) \times l$ matrix of entries $0$, $M_2$ is the $l \times l$ diagonal matrix, formed by the right $l$ columns in the matrix above, and $M_2$ is $(l+1)\times (l+1)$ given by the first $(l+1)$ columns of the vectors $v_3, v_7, ..., v_0$. We know that $$\det M =  - \det M_2 \det M_3.$$

We write out $M_3$ explicitly with the choice of $v_0 = (0, 1, 0, ..., 0)$:
\begin{equation*} M_3 = 
\arraycolsep=2pt
\medmuskip = 0.3mu
\smaller
\begin{bmatrix}
b_2 & -1 & b_3 - 2 b_2 & \ldots & b_l - (l-1) b_2  &  b_{l+1} - l b_2 \\
b_3 - b_2 & -1 & (b_4 - b_2)-2 (b_3 - b_2) & \ldots &  (b_{l+1}-b_2)-(l-1)(b_3 - b_2) & (b_{l+2}-b_2)-l (b_3 - b_2)\\
\vdots \\
b_{l+1} - b_l & -1 & (b_{l+2}-b_l)-2 (b_{l+1}-b_l) & \ldots & & (b_{2l}-b_l)-l (b_{l+1}-b_l)\\
0 & 1 & 0 & \ldots & 0 & 0
\end{bmatrix}
\end{equation*}

Note that if we interchange the first two columns of the matrix, the appearance of the term $b_{l+1}$ is only on the off-diagonal and the second last row. Suppose that $b_{l+1} = \beta$ and all other $b_j = 0$, then the matrix (after interchanging the first two columns) becomes: 

\begin{equation*} M_\beta =
\begin{bmatrix}
-1 & 0 & 0 & \ldots & 0 & \beta\\
-1 & 0 & 0 & \ldots & \beta & 0 \\
\vdots  & & &\reflectbox{$\ddots$} & & \vdots\\
-1 & 0 & \beta   & \ldots & 0 & 0\\
-1 & \beta & -2 \beta  & \ldots & -(l-1) \beta & -l \beta \\ 
1 & 0  & 0& \ldots & 0 & 0 
\end{bmatrix}
\end{equation*}
which is easily seen to be invertible, since we can add the last row to each row above and then add $(l+1)-i$ copies of the $i^{th}$ row to the second last row for all $ 1 \le i \le l-1$, and then obtain a off-diagonal matrix. Now we can simply set $b_{i+1} - b_{i} = \epsilon$ for all $i \le l -1$ with sufficiently small $\epsilon > 0$, so $\det M_2 \ne 0$. Let $M_3'$ be $M_3$ with first two columns interchanged, then $M_3'$ can be obtained by adding a matrix with $\epsilon$-small entries to $M_\beta$. If we make $\beta$ sufficiently large, it is clear that $\det M_3 \ne 0.$ 

This proves the lemma.
\end{proof}

\begin{definition}
We call a parallel lifting of a polygon \emph{good} if it satisfies the requirements in Lemma \ref{lemma:lifting_into_general_positions}.  
\end{definition}

In particular, Lemma \ref{lemma:lifting_into_general_positions} tells us that, if $\mathscr L$ is the space of all parallel liftings of all polygons, then the subspace $\mathscr L' \subset \mathscr L$ of the good parallel liftings is a generic subset. This is because $\mathscr L'$ is a subset defined by the non-vanishing of certain algebraic equations, hence Zariski open, which is generic as long as it is non-empty.  

\begin{notation}
To ease notation for the rest of the section, we denote $$J_1 = \widetilde A_1, \quad J_3 = \widetilde A_3, \quad  ..., \quad J_{2n-3} = \widetilde A_{2n -3}$$ if the lifting is good. In which case, we also label the prisms by $$T_2 = J_1 J_3, \quad T_4 = J_3 J_5, \quad  ..., \quad T_{2n-4} = J_{2n-5} J_{2n-3}.$$ 
\end{notation}

Note that each $T_k$ consists of $n$ parallel lines $(t^k(1), t^k(3), ..., t^k(2n-1))$, where the order of the lines are inherited from the order of corresponding pairs $J_{k-1} J_{k+1}$, in other words,  $t^k(j) = \overline{ J_{k-1} (j) J_{k+1}(j)}$, where $J_{k-1}(j)$ is the $j^{th}$ element in $J_{k-1} = \widetilde A_{k-1}$. 

Now we define more terminologies. 

 \begin{definition}
Let $T= (t_1(1), t_1 (3), ..., t_1 (2n-1))$ be a prism where $t_1(2j -1)$ represents a line. Following Schwartz \cite{Schwartz2}, we define a structure called \emph{cyclic skeletons} inductively as follows: First set $\Sigma_1 T = T$. Define $t_k (j)$ to be the subspace of $\R^n$ spanned by subspaces  $t_{k-1}(j-1)$ and $t_{k-1}(j+1)$, and then set $\Sigma_k T = (..., t_k (j), ... )$. In other words, $t_k (j)$ is spanned by $k$ consecutive lines, and thus has dimension $k$. It is clear from definitions that $$t_{k-1}(j) = t_k (j-1) \cap t_k (j+1).$$ 
\end{definition}

The cyclic skeletons formalize the description of ``faces'' of the prisms. For a prism $T$, the cyclic $k$ skeleton $\Sigma_k T$ is a set of $n$ elements, where each element is a $k$-dimensional face of $T$ in $\R^n$.  As mentioned in the outline of the proof, we use the notion of the skeletons in the following crucial way: We produce a codimension $d$ flat by intersecting $d$ hyperplanes $|J_i|, |J_{i+2}|, ..., |J_{i+2d -2}|$ (thanks to Lemma \ref{lemma:lifting_into_general_positions}). Then  we intersect this $d$ flat with each one of the faces in the cyclic $d$ skeleton of one of the prisms, to produce $n$ points in $\R^n$ (this requires the notion of a \emph{perfect lifting}, which says precisely all such intersections are transverse). 
\vspace*{0.1cm}

Before we prove that all the intersections are indeed transverse, we first follow \cite{Schwartz2} and formalize the notion of intersecting the $d$ flat with cyclic skeletons, and then project these points to the plane. As predicted, it turns out that these projected points recover the orbits of the vertices under iterations of the pentagram map.

\begin{definition}
Let $W$ be a codimension $j$ flat and $T$ be a prism. We say that $W$ \emph{slices} $T$ if $W \cap \Sigma_j T$ consists of $n$ distinct points and $W \cap \Sigma_{j+1} T$ consists of  distinct lines (as long as $j < n-1$).   Suppose that $W$ slices $T$, we define $W_T = \pi (W \cap \Sigma_j T)$, which is the projection of the intersection points, thus $W_T$ consists $n$-point. 
\end{definition}

\begin{definition}
Let $W$ and $W'$ be two distinct flats, we say that the pair $(W, W')$ slices $T$ if $\dim (W \cap W') +1 = \dim (W) = \dim (W')$ and if $W, W', W \cap W'$ all slice $T$. 
\end{definition}

The next lemma says that one iteration of the pentagram map can be obtained by taking intersections of two joints with the prism and then projecting down. 

Before we state the lemma, we first define the notion of the mating process, again following \cite{Schwartz2}. 

\begin{definition}
Let $X = (x_1, x_3, ..., x_{2n-1})$ and $Y= (y_1, y_3, ..., y_{2n -1})$ be 2 $n$-points. We construct $Z = X * Y$ where $Z = (z_2, z_4, ..., z_{2n-2}, z_{2n})$ by defining $$z_j = \overline{x_{j-1} x_{j+1}} \cap \overline{y_{j-1} y_{j+1}}$$ 
Note that labelling-wise, $Z = (z_2, z_4, ..., z_{2n})$. Now let  $X_1 = (X_{1,1} , X_{1,3} , ..., X_{1,2m-1}) $ be a sequence of $n$-points, such that for each pair $X_{1, 2k-1}$ and $X_{1, 2k+1}$, their product $X_{2, 2k} = X_{1, 2k-1}* X_{1, 2k+1}$ is well defined,  then the sequence $X_1$ produce a new sequence $X_2 = (X_{2,2} , X_{2,4} , ..., X_{2,2m-2}) $, and we denote $X_1 \rightarrow X_2$ whenever the process is well defined. This progression could carry on to obtain: $$X_1 \rightarrow X_2 \rightarrow ... \rightarrow X_m$$ where $X_m$ is one set of $n$-points, provided that each step is well defined. Schwartz calls this process the mating process on a sequence of $n$-points. 
\end{definition}

In our setting, it is clear from definition that, if we take $X_1 = (A_1, A_3, A_5, ..., $ $A_{2n - 3})$, then the underlying set of $X_2$, in other words, by taking the union of sets in $X_2$ (not a multi-set), gives all the vertices of $T(P)$, the first iteration of the pentagram map. Similarly $X_i$ provides all the vertices of $T^{i-1}(P)$ for $i \le n - 2$. In the last step, $X_{n-1}$ gives precisely half of the points in $T^{n-2} (P)$. Our goal is to prove that, if $P$ satisfies conditions in Theorem \ref{T002}, then the points in $X_{n-1}$ are collinear, where the line that goes through them also goes through the center of mass $\mathscr C (P)$. 


\begin{lemma}[Schwartz, Lemma 3.2 in \cite{Schwartz2}]{\label{lemma:intersection_gives_mating_process}}  Suppose $(V, V')$ slices $T$, and the offspring $(V)_T * (V')_T$ is well defined, then $(V \cap V')_T = (V)_T * (V')_T$. 
\end{lemma}

Let us denote $H_{1,k} = |J_k|$ for all appropriate choices of $k$, where $J_k$ is defined as above. We inductively define $$H_{g,k} = H_{g-1, k-1} \cap H_{g-1, k+1}; \quad g = 2, 3, ..., n-1, k \in [g, 2(n-1) -g].$$
Each $H_{g,k}$ is in fact intersection of $g$ hyperplanes. For example, if $g$ is odd, then $$H_{g,k} =  |J_{k-(g-1)}| \cap ... \cap  |J_{k-2}| \cap |J_k| \cap |J_{k+2}| \cap ... \cap |J_{k+(g-1)}|,$$ and a similar equation holds for $g$ even. 

Consider a good parallel lifting of a generic polygon $P$. Lemma \ref{lemma:lifting_into_general_positions} shows that all hyperplanes $|J_i|$ are in general position, so each $H_{g,k}$ has codimension $g$. 


\begin{lemma}\label{lemma:intersection_with_different_prisms} Let $g$ and $k$ be integers of the same parity, then  $$H_{g, k} \cap \Sigma_g T_{k-g}  = H_{g, k} \cap \Sigma_g T_{k-g+2}  = \ldots  = H_{g, k} \cap \Sigma_g T_{k+g}.$$
In particular, if both $g$ and $k$ are odd, then $$H_{g, k} \cap \Sigma_g T_{k-1}  = H_{g, k} \cap \Sigma_g T_{k+1};$$
if both $g, k$ are even, then 
$$H_{g, k} \cap \Sigma_g T_{k-2}  = H_{g, k} \cap \Sigma_g T_{k} =H_{g, k} \cap \Sigma_g T_{k+2}. $$ 
\end{lemma}
\begin{proof}

We prove the lemma by pairwise comparison. Namely, we show that for all even integers $l$ such that $k-g \le l \le k+g - 2$ (where all indices are labeled cyclically), we have that $$H_{g, k} \cap \Sigma_g T_{l}  = H_{g, k} \cap \Sigma_g T_{l+2}.$$ It is clear that the claim implies the lemma.  First note that by construction the following equation $$\Sigma_{*} T_{l} \cap H_{1, l+1} = \Sigma_{*} T_{l+2} \cap H_{1, l+1} $$ holds for all even $l$. Now when we restrict $l$ to the interval above, in other words, $l+1 \in [k-g+1, k+g-1]$, then $H_{g, k} \in H_{1, l+1}$. This is clear from the definition of $H_{g, k}$, which is the intersection of $g$ hyperplanes $$H_{1, k-g+1} \cap H_{1, k-g+3} \cap ... \cap H_{1, k+g-1}.$$
Therefore, combining the two equations above, we have that  $$H_{g, k} \cap \Sigma_g T_{l} = H_{g, k} \cap H_{1, l+1} \cap \Sigma_g T_{l} =  H_{g, k} \cap H_{1, l+1} \cap \Sigma_g T_{l+2} = H_{g, k} \cap \Sigma_g T_{l+2}.$$
This proves the case where both $g$ and $k$ are odd, and a similar argument (which possible shifting of indices) shows the case where both $g, k$ are even.
\end{proof}

Now we formalize the definition of transverse intersections between a codimension $g$ flats $H_{g, k}$ with certain cyclic skeletons. In particular, we define the notion of a polyjoint being \emph{fully sliced}, namely, all the possible intersections we consider are transverse -- so they produce points instead of  higher dimensional flats or the empty set. 
 
\begin{definition}
Let $\Omega = (J_1, J_3, ..., J_{2n-3})$ defined as above, where $J_k = \widetilde A_k$ and the lifting is a good lifting. We say that $\Omega$ is \emph{fully-sliced} if for all $g \le n-1$, $H_{g,k}$ slices $T_h$ for all relevant $h$ such that $|h-k| \le 1$. From Lemma \ref{lemma:intersection_with_different_prisms} we know that when $g, k$ are both odd, $H_{g,k}$ slices $T_{k-1}$ if and only if it slices $T_{k+1}$, so the notion of fully slices is well defined. If $\Omega$ is fully sliced, we define 
$$X_{g,k} = \pi (H_{g, k} \cap \Sigma_g T_h)$$ where $h \in [k-1, k+1]$ and define $$X_g = (X_{g,g}, X_{g, g+2}, ..., X_{g, 2(n-1)-g}).$$
\end{definition}

By the construction it is clear that $X_1 = (A_1, A_3, ..., A_{2n-1}).$

\begin{lemma}[Schwartz, Lemma 3.3 in \cite{Schwartz2}]{\label{lemma:fully_sliced_mating_process}} Suppose $\Omega = (J_1, ..., J_{2n-3})$ is fully sliced, and the mating process is defined on all $X_i$, then $$X_1 \rightarrow X_2 \rightarrow ... \rightarrow X_{n-1},$$ in other words, each $X_{i+1}$ is obtained from $X_i$ by the mating process defined in the beginning of the section.
\end{lemma} 

\begin{definition}
We say that a good lifting is \emph{perfect} if the polyjoint $\Omega = (J_1, J_3, ..., J_{2n-3})$ is fully sliced. 
\end{definition}
 
 
\begin{lemma}{\label{lemma:perfect_lifting_existence}} There subspace of all axis-aligned polygons such that there exists a perfect parallel lifting is a generic subset of the space of all axis-aligned polygons.
\end{lemma}
\begin{proof}
This follows from the proof of lemma 3.4 in \cite{Schwartz2}. 
\end{proof}

As we remarked previously, the mating process is a generalization of the pentagram map in the sense that iterations of the pentagram map on the polygon $P$ can be obtained through the mating process of $(A_1, A_3, ..., A_{2n-3})$:

\begin{lemma}\label{lemma:mating_gives_pentagram_map}
Let $Y_1 = (A_1, A_3, A_5, ..., A_{2n - 3})$ where each $A_i$ is defined as in the beginning of this section, then the mating process $$Y_1 \rightarrow Y_2 \rightarrow ... \rightarrow Y_{d+1}$$ is well defined for $d$ steps if and only if the $T^{d}$ is well defined on $P$. In addition, assume that  mating process $$Y_1 \rightarrow Y_2 \rightarrow ... \rightarrow Y_{n-1}$$ is well defined,  then the underlying set of $Y_i$ gives all the vertices of $T^{i-1}(P)$ for $i \le n - 2$. In the last step, $Y_{n-1}$ gives precisely half of the points in $T^{n-2} (P)$. 
\end{lemma}

\begin{proof}
This follows directly from the definition of the mating process and the construction of the sequence of $n$-points $(A_1, A_3, ..., A_{2n-3})$. Note that, if $T^{n-2}$ s well defined on $P$, and let $Q = T^{n-2}(P)$, then the $n$ points in the unique sequence of $n$-points in $Y_{n-1}$ are evenly spaces in $Q$, in other words, they are either the $1^{st}, 3^{rd} ,... $ points or the $2^{nd}, 4^{th}, ...$ points in $Q$, depending on the parity of $n$. This follows from carefully keeping track of relevant indices.
\end{proof}

We are ready to prove the main theorem in this section. 
\begin{proof}[Proof of Theorem \ref{T002}]

First we prove that if $P$ is polygon such that $T^{n-1}$ is well defined on $P$ and $P$ has a perfect lifting, then the point of collapse $T^{n-1}(P)$ is its center of mass $\mathscr C (P)$.

Let $A_1, A_3, ..., A_{2n-3}$ be as above, and $\widetilde A_1, \widetilde A_3, ..., \widetilde A_{2n-3}$ be lifted through a perfect lifting from $A_1, A_3, ..., A_{2n-3}$.  Let $Y_1 = (A_1, A_3, ..., A_{2n-3})$. Since $T^{n-1}$ is well defined on $P$, by Lemma \ref{lemma:mating_gives_pentagram_map}, we know that the mating process is well defined on $Y_1$ for $n-1$ steps, so we have 
$$Y_1 \rightarrow Y_2 \rightarrow ... \rightarrow Y_{n-1}.$$

Define $H_{1,1}, H_{1,3}, ..., H_{1,{2n-3}}$ as in Lemma \ref{lemma:intersection_with_different_prisms}, and define $X_1, X_2, ..., X_{n-1}$ as above Lemma \ref{lemma:fully_sliced_mating_process}. Since the lifting is assumed to perfect, that is to say, $H_{1,1}, ..., $ $H_{1,2n-3}$ are in general position and the joint  $\widetilde A_1, \widetilde A_3, ..., \widetilde A_{2n-3}$  is fully sliced, by Lemma \ref{lemma:fully_sliced_mating_process}, we know that $$X_1 =  Y_1 = (A_1, A_3, ..., A_{2n-3})$$ and  $$X_1 \rightarrow X_2 \rightarrow ... \rightarrow X_{n-1}.$$ 
Therefore, $Y_i = X_i \text{ for all } i \le n-1.$ Hence $$Y_{n-1} = X_{n-1} = (X_{n-1, n-1}) = \pi (H_{n-1, n-1} \cap \Sigma_{n-1} T_h),$$ which consists of $n$ points. Note that $H_{n-1, n-1}$ is a line since it is the intersection of $(n-1)$ hyperplanes in general position. Hence the $n$ points $\widetilde X_{n-1} = H_{n-1, n-1} \cap \Sigma_{n-1} T_h $ in space are collinear, and all lie on the line $L_{n-1} = H_{n-1, n-1}$. Let $l_{n-1} = \pi (L_i)$ be the projection into $\R^2$, and it is clear that $l_{n-1}$ goes through all points in $X_{n-1} = \pi(\widetilde X_{n-1})$. Since $X_{n-1} = Y_{n-1}$, we know that $l_{n-1}$ goes through $Y_{n-1},$ in other words, half of the points (the indices of which are of the same parity) in $T^{n-2}(P).$. 

We want to show that the line $l_{n-1}$ goes through $\mathscr C (P)$.  Let $ C_i $ be the centroid of the $n$-vertices in $J_i = \widetilde A_i$ as in Lemma \ref{lemma:parallel_lifting_of_centroids}, so $$C_1 = C_3 = ... = C_{2n-3}.$$ Denote this point in $\R^n$ by $C$. It is clear that the centroid $C_i  \in |J_i| = H_{1, i}$, so the point $C \in H_{1, j}$ for all $j = 1, 3, ..., 2n-3$. By the definition of $H_{n-1, n-1}$, it is clear that $C \in H_{n-1, n-1} = L_{n-1}$. Again by Lemma \ref{lemma:parallel_lifting_of_centroids} $$\mathscr C (P) = \pi (C) \in \pi (L_{n-1}) =  l_{n-1}.$$  
This implies that $T^{n-1}(P) = \mathscr C (P)$, since the line $l_{n-1}'$ that goes through the other $n$ points in $T^{n-2}(P)$ also goes through $\mathscr C (P)$ by symmetry. Since $T$ is well defined on $T^{n-2}(P)$ by assumption, we know that $l_{n-1}$ and $l_{n-1}'$ do not degenerate. Hence they have to intersect, at their common point $\mathscr C (P)$. The intersection $l_{n-1} \cap l_{n-1}'$ is the degenerate point $T^{n-1}(P)$ by the definition of the pentagram map, therefore the point $T^{n-1}(P)$ coincide with the center of mass.

Let $S$ be the space of all axis-aligned polygons. Theorem \ref{T001} says that the subspace of axis-aligned polygons such that $T^{n-1}$ is well defined on $P$ is a generic subspace of $S$.  Lemma \ref{lemma:lifting_into_general_positions} and Lemma \ref{lemma:perfect_lifting_existence} guarantee that the subspace of axis-aligned polygons that have a perfect lifting is again a generic subspace of $S$.  So generically, for axis-aligned $2n$-gon $P$, we have $$T^{n-1} (P) = \mathscr C (P).$$

\end{proof}


\section{{Higher pentagram map}}\label{sec:higherpentagram} 

In this section we generalize Theorem \ref{T002} to a certain family of higher pentagram maps, defined in \cite{GSTV} by Gekhtman, Shapiro, Tabachnikov and Vainshtein. We will first present the definition of the \emph{corrugated polygons} in $\P^m$ and the \emph{corrugated pentagram map} $T_m$, as a generalization of the pentagram map into $\P^m$ for $m \ge 3$. As in Section \ref{sec:pentagram}, we restrict our attention to axis-aligned polygons in $\R^m \subset \P^m$. The main objective in this section is to prove Theorem \ref{T003}, which states that a generic axis-aligned $mn$-gon $P \subset \R^m$ collapses to a point under the $n-1$ iterations of $T_m$, which equals its center of mass.

\subsection{The corrugated polygons and higher pentagram map}

\begin{definition}[GSTV $'11$ \cite{GSTV}]
Let $V = V_1V_2  ...  V_{k} $ be a polygon in $\P^m$, we say that $V$ is corrugated if, for every $i$, the vertices $V_i, V_{i+1}, V_{i+m}$ and $V_{i+m+1}$  span a projective plane. 
\end{definition} 

On the space of corrugated polygons, the successive pairs of  $m$-diagonals intersect by definition (which is not true for general polygons), so we have the following natural generalization of the pentagram map on the space of corrugated polygons:

\begin{definition}[GSTV $'11$ \cite{GSTV}]
Let $V \subset \P^m$ be a corrugated polygon, the (corrugated) higher pentagram map $T_m$ on $V$ is defined by taking the intersections of successive pairs of $m$-diagonals as vertices of $T_m(V)$, labelled according to the cyclic order on the vertices of $V$. 
\end{definition}

\begin{lemma}[GSTV $'11$, Theorem $5.2 (i)$ \cite{GSTV}] \label{lemma:GSTV_corrugated}
Let $V \subset \P^m$ be a corrugated polygon, then $T_m (V)$ is also corrugated.  
\end{lemma}

The lemma says that we can in fact define iterations of corrugated pentagram maps on corrugated polygons, as long as resulted polygons do not have degenerate vertices.

\begin{remark} 
There are other generalizations of the pentagram map into higher dimensional projective spaces. For example, Khesin and Soloviev defined the notion of the dented pentagram map using intersections of certain hyperplanes in \cite{KS}. For the generalization of our result Theorem \ref{T003}, the most natural choice of higher pentagram map is the one by taking higher diagonals, namely  $T_m$. We call it the \emph{corrugated pentagram map} to distinguish from the dented map, since both authors originally named their maps ``the higher pentagram map''.  
\end{remark}

Now we define the corresponding notion of axis-aligned polygon. 

\begin{definition}
A closed polygon $V = V_1 V_2 ... V_k$ is axis aligned in $\R^m$ if the edge $V_iV_{i+1}$ is parallel to the $x_j$ axis where $j = i \mod m$. In other words, we have $V_1V_2$ in the direction of $x_1$ axis, $V_2V_3$ in the direction of $x_2$ axis, and so forth.
\end{definition}

\begin{figure}
\begin{center}
\includegraphics[scale=0.55]{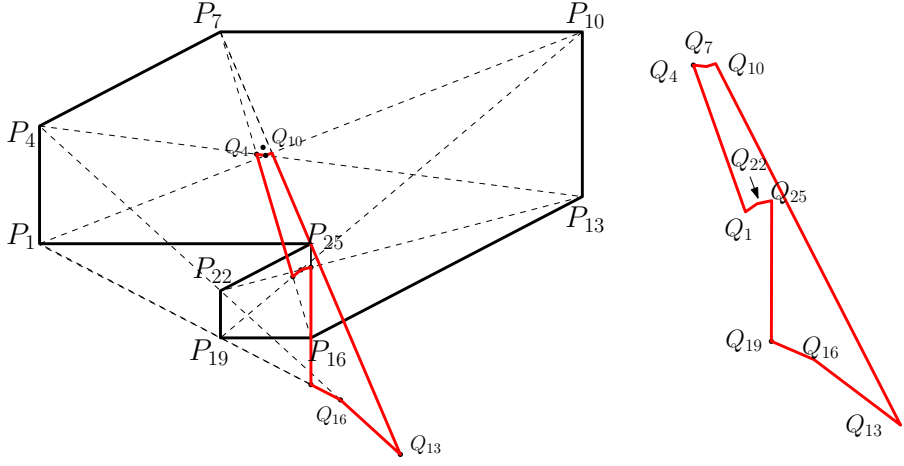} 
\end{center} 
\caption{\label{figure:3d_pentagram_example_1} An axis-aligned $9$-gon $P$ under the 3d pentagram map $T_3$.} 
\end{figure}

\begin{remark}
 It is clear that if a $k$-gon $V = V_1 V_2 ... V_k$ is axis-aligned in $\R^m$, then $m |k$.  In the rest of the section, we will denote an axis-aligned polygon in $\R^m$ by $P$, where $P$ is an $mn$-gon. 
\end{remark}

\begin{definition}
Let $P$ be an axis-aligned $mn$-gon in $\R^m \subset \P^m$, we define the center of mass $\mathscr C(P)$ of $P$ to be the arithmetic mean of the vertices in $P$ (regarded as vectors in $\R^m$). 
\end{definition}

\begin{remark}
As in the introduction, we could also define axis-aligned $mn$-gon as one with corresponding edges going through $n$ points, and define the center of mass in a projectively invariant way: by first applying a projective transformation so that $P$ has edges parallel to axis, then taking the centroid and transforming back. Again, for our purposes, it suffices to consider the case described in the above definitions. 
\end{remark}

We observe that an axis-aligned polygon in $\R^m \subset \P^m$ is corrugated by construction, we state this as a lemma:
\begin{lemma}
Let $P$ be an axis-aligned $mn$-gon in $\R^m$, then $P$ is a corrugated polygon in $\R^m$. Hence $T_m$ is defined. 
\end{lemma}

Now we describe the labelling convention for an axis-aligned $mn$-gon $P$, which is compatible with the case where $m=2$.
\begin{notation}
Let $P$ be an axis-aligned $mn$-gon in $\R^m \subset \P^m$, we label the $mn$ vertices of $P$ by $$P_1,\, \,  P_{m+1},\, \, P_{2m+1},\, \, ...,\, \, P_{(mn-1)m+1}.$$ Thus an $m$-diagonal has the form of $P_{j} P_{m^2+j}$.
\end{notation}

\begin{notation}
 Let $T_m (P) = Q$, then we label vertices of $Q$ by $$Q_{k}, \, \, Q_{m+k}, \, \, Q_{2m+k}, \, \, ...,  \, \, Q_{(mn-1)m+k}$$ where $k = \frac{m^2 +m +2 }{2}$. In particular, we let $Q_{k} = Q_{(m^2 +m +2 )/2}$ to be the intersection $P_1 P_{m^2+1} \cap P_{m+1} P_{m^2+m+1}$, and the index of the vertex in $Q$ comes from averaging the indices of the relevant $4$ vertices in $P$.
\end{notation}

\begin{example}
For an example of an axis-aligned $9$-gon in $\R^3$ with the specified labelling convention, see Figure \ref{figure:3d_pentagram_example_1}. The dotted lines represent $3$-diagonals and the image of $T_3(P)$ is also presented in the figure. \\

\end{example}

\subsection{The point of collapse under the corrugated pentagram map}\label{subsection:3.2}  
We break the proof of Theorem \ref{T003} into two parts. First we prove the theorem under the assumption that $n \ge m$, namely:

\subsubsection{The case when $m \le n$} 
\begin{theorem}\label{T003:n>m} 
Let $P$ be a generic axis-aligned $mn$-gon in $\R^m$, where $ n \ge m$, then $T_m^{n-1} (P)$ collapses to the center of mass of $P$, in other words, $$T_m^{n-1}(P) = \mathscr C (P).$$ 
\end{theorem}

\begin{example}
Figure \ref{figure:3d_pentagram_example_2} provides an example of the statement with $n=m=3$, in which case $T_3^2 (P) = \mathscr C (P).$ We omit the labelling of the vertices for the sake of presentation. 
\begin{figure}
\begin{center}
\includegraphics[scale=0.5]{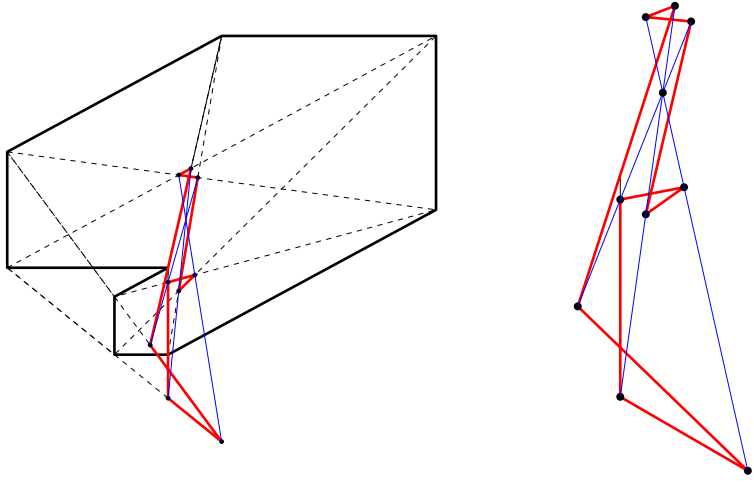} 
\end{center} 
\caption{\label{figure:3d_pentagram_example_2}  The point of collapse of an axis-aligned $9$-gon in $\R^3$ under  $(T_3)^2$.}
\end{figure}
\end{example}

The proof of Theorem \ref{T003:n>m} is analogous to the proof of Theorem \ref{T002}, and we shall provide less details. 

Suppose that $P = P_1 P_{m+1} P_{2m+1} \, ...\,  P_{(mn-1)m+1} \subset \R^m \subset \R^n$ is an axis-aligned $mn$-gon. We start with $n-1$ sequences of $n$-points in $\R^m$, constructed similarly as in Section \ref{sec:pentagram}: 

\begin{alignat*}{2}
 & A_1 &&= \left\{P_1(1),\,\,  P_{m^2+1}(1), \,\,  ..., \,\,  P_{ (n-2)m^2 +1 }(1),  \,\, P_{ (n-1)m^2 +1 }(1)\right\}   \\ 
 & A_{m+1} &&  = \left\{P_{m+1} (m+1), \,\,  P_{m^2+ m+1} (m+1) ,\,\,  ..., \,\,  P_{(n-1)m^2 + m+1}(m+1) \right\} \\
 & A_{2m+1} && =\left\{P_{2m+1} (2m+1), \,  P_{m^2+ 2m+1} (2m+1) , \, ..., \,  P_{(n-1)m^2 +2m+1}(m+1) \right\} \\ 
 & && \ldots  \\ 
 &  A_{(n-2)m+1} &&  =\left\{P_{(n-2)m+1} ((n-2)m+1), \,  ..., \,  P_{(n-1)m^2 + (n-2)m+1}((n-2)m+1) \right\}  
 \end{alignat*}  
\vspace*{0.15cm}

We lift these points to $\R^n$ by adding $x_{m+1}, ..., x_{n}$ coordinates to each point above. As before, a \emph{parallel lifting} is lifting such that for all $l = 1, ..., n$,  the $l^th$ point of each sequence $A_j$ (for appropriate $j$) has the same $x_{m+1}, ..., x_{n}$ coordinates, while the definitions of \emph{prisms} and \emph{polyjoints}  remain the same as in Section \ref{sec:pentagram}.


\begin{lemma}{\label{lemma:general_parallel_lifting_into_joint}} 
A parallel lifting lifts the sequence of $n$-points defined above $A_1, A_{m+1}, $  $ ...,  A_{(n-2)m + 1}$ to a polyjoint $(\widetilde A_1, \widetilde A_{m+1}, ..., \widetilde  A_{(n-2)m + 1} ).$
\end{lemma}

\begin{lemma}{\label{lemma:general_parallel_lifting_of_centroids}} Suppose that a polyjoint $(\widetilde A_1, \widetilde A_{m+1}, ..., \widetilde  A_{(n-2)m + 1} ) $ is lifted from $A_1, A_{m+1}, $  $ ...,  A_{(n-2)m + 1}$ by a parallel lifting. Let $ C_i \in \R^n$ be the centroid of the $n$-vertices in $\widetilde A_i$ for each $i = 1, m+1, ..., (n-2)m+1$.  Then
\begin{equation*}  C_1 = C_{m+1} = ... = C_{(n-2)m+1}; \quad \text{and }
 \pi (C_1) = ... = \pi (C_{(n-2)m+1}) = \mathscr C (P)
\end{equation*} where $\pi: \R^n \rightarrow \R^m$ is the projection into the first $m$ coordinates. 
\end{lemma}
\begin{proof}[Proof of Lemma \ref{lemma:general_parallel_lifting_into_joint} and \ref{lemma:general_parallel_lifting_of_centroids}]

Both lemmas are clear by carefully keeping track of definitions, and the details are similar to the proofs of Lemma \ref{lemma:parallel_lifting_into_joint} and \ref{lemma:parallel_lifting_of_centroids}.
\end{proof}

\begin{lemma}{\label{lemma:general_lifting_into_general_positions}} There exists an axis-aligned polygon with a parallel lifting of $A_1, ...,$ $A_{(n-2)m+1}$ such that all subspaces $|\widetilde A_1|, |\widetilde A_{m+1}|, ..., |\widetilde  A_{(n-2)m + 1} |$ are hyperplanes in general position. 
\end{lemma}
\begin{proof} The proof is similar to the proof of Lemma \ref{lemma:lifting_into_general_positions}, except that we have more freedom to choose coordinates for the $mn$ points,  so the criterion of the determinant of a similar matrix as in Lemma \ref {lemma:lifting_into_general_positions} being nonzero is easier to satisfy.
\end{proof}

The rest of the construction is again similar to the $2$-dimensional case. We adjust the notion of mating process and slicing  to $\R^m$, while while keeping the same definitions of cyclic skeleton, perfect lifting, etc.  It is not surprising that similar results in Section \ref{sec:pentagram} extend naturally to the higher dimensional case. 

\begin{definition}[higher mating]
Let $X = (x_1, x_3, ..., x_{2n-1}) $ and $Y= (y_1, y_3, ..., y_{2n -1})$ be 2 sequences of $n$ points in $\R^m$. We say that $X $ and $Y$ are relatively corrugated, or a corrugated pair, if the $4$ points $x_{i_1}, x_{i+1}, y_{i-1}, y_{i+1}$ span a plane in $\R^m$ for all $i$. If $X, Y$ are a corrugated pair, then  we can construct $Z = X * Y$ by the same rule as before by defining 
$$z_j = \overline{x_{j-1} x_{j+1}} \cap \overline{y_{j-1} y_{j+1}}.$$ 
Let  $X_1 = (X_{1,1} , X_{1,3} , ..., X_{1,2m-1}) $ be a sequence of $n$-points in $\R^m$, such that each pair of successive $n$-points $X_{1, 2k-1}$ and $X_{1, 2k+1}$ are relatively corrugated, then we form a new sequence   $X_2 = (X_{2,2} , X_{2,4} , ..., X_{2,2m-2})$ as in the mating process, where $X_{2,j} = X_{1,j-1} * X_{1,j+1}$. Lemma \ref{lemma:GSTV_corrugated} says that each pair of 2 successive $n$-points in $X_2$ is again relatively corrugated. As before, we denote $X_1 \rightarrow X_2$ whenever the process is well defined. This process is called the higher mating process. 
\end{definition}

\begin{definition}[higher slicing]
We have the same definition for slicing, pairwise slicing and fully slicing as in Section \ref{sec:pentagram}. The only notable change is that our projection $\pi: \R^n \rightarrow \R^m$ is projecting to $\R^m$ instead of $\R^2$, again by dropping the $x_{m+1}, ..., x_n$ coordinates.  We retain notations from Section \ref{sec:pentagram}.
\end{definition}

Following the proofs line by line of Lemma \ref{lemma:intersection_gives_mating_process} to Lemma \ref{lemma:perfect_lifting_existence}, we note that the statements of   Lemma \ref{lemma:intersection_gives_mating_process}  to  \ref{lemma:perfect_lifting_existence} remain the same (under corresponding interpretations of notations in $\R^m$ instead of $\R^2$) in higher dimensions. While the  proof of the corresponding lemma of Lemma \ref{lemma:mating_gives_pentagram_map} is word by word similar as before, certain notations need to be adjusted.

\begin{lemma}\label{lemma:general_mating_gives_pentagram_map}
Let $Y_1 = (A_1, A_{m+1}, A_{2m+1}, ..., A_{(n-2)m+1})$ where each $A_i$ is defined as above, then the higher mating process $$Y_1 \rightarrow Y_2 \rightarrow ... \rightarrow Y_{d+1}$$ is well defined for $d$ steps if and only if the $T^{d}$ is well defined on $P$. In addition, assume that  mating process $$Y_1 \rightarrow Y_2 \rightarrow ... \rightarrow Y_{n-1}$$ is well defined,  then the underlying set of $Y_i$ gives corresponding vertices of $T_m^{i-1}(P)$ for $i \le n - 1$. 
\end{lemma}

Now we are ready to prove the main theorem for corrugated pentagram maps.


\begin{proof}[Proof of Theorem \ref{T003:n>m}] 

Let $P = P_1 P_{m+1} P_{2m+1} \, ...\,  P_{(mn-1)m+1} \subset \R^m$ be an axis-aligned $mn$-gon such that $T_m^{n-1}$ is well defined on $P$. 

Let $A_1, A_{m+1},  ...,  A_{(n-2)m + 1}$ be defined as above, and $(\widetilde A_1, \widetilde A_{m+1}, ..., \widetilde  A_{(n-2)m + 1} )$ a polyjoint lifted from $A_1, A_{m+1}, ...,  A_{(n-2)m + 1}.$ 

Let $Y_1 = (A_1, A_{m+1},  ...,  A_{(n-2)m + 1})$. Since $T_m^{n-1}$ is well defined on $P$, the higher mating process is well defined on $Y_1$ for $n-1$ steps, so we have 
$$Y_1 \rightarrow Y_2 \rightarrow ... \rightarrow Y_{n-1}.$$

Define hyperplanes $H_{1,1}, H_{1,m+1}, ..., H_{1,{(n-2)m-1}}$ as in Section \ref{sec:pentagram} (see discussion above Lemma \ref{lemma:intersection_with_different_prisms}), and define $X_1, X_2, ..., X_{n-1}$ as in the definition proceeding Lemma \ref{lemma:fully_sliced_mating_process}. By definition of perfect lifting, we know that $H_{1,1}, ..., $ $H_{1,(n-2)m+1}$ are in general position and the joint  $(\widetilde A_1, \widetilde A_{m+1}, ..., \widetilde A_{(n-2)m+1})$ is fully sliced. The analogous statement of Lemma \ref{lemma:fully_sliced_mating_process} tells us that $$X_1 =  Y_1 = (A_1, A_{m+1},  ...,  A_{(n-2)m + 1}) \quad \text{ and } \quad  X_1 \rightarrow X_2 \rightarrow ... \rightarrow X_{n-1}.$$ 
Therefore, $Y_i = X_i \text{ for all } i \le n-1.$ In particular $$Y_{n-1} = X_{n-1} = (X_{n-1, n-1}) = \pi (H_{n-1, n-1} \cap \Sigma_{n-1} T_h),$$ which consists of $n$ points. Similar to before,  $H_{n-1, n-1}$ is a line as the transverse intersection of $(n-1)$ hyperplanes. 

The $n$ points $\widetilde X_{n-1} = H_{n-1, n-1} \cap \Sigma_{n-1} T_h $ in space are therefore collinear, through a line $L_{n-1} = H_{n-1, n-1} \subset \R^n$. Let $l_{n-1} = \pi (L_i)$ be the projection into $\R^m$, and it is clear that $l_{n-1}$ goes through all points in $\pi(\widetilde X_{n-1})  = X_{n-1} = Y_{n-1}$.

Let $ C_i $ be the centroid of the $n$-vertices in $\widetilde A_i$ as in Lemma \ref{lemma:general_parallel_lifting_of_centroids}, so $$C_1 = C_{m+1} = ... = C_{(n-2)m+1} = C.$$  Therefore by a similar argument in the proof of Theorem \ref{T002}, the line $l_{n-1}$ goes through $\mathscr C (P) = \pi (C)$.   Again by symmetry, we know that the other $m-1$ similarly constructed lines also go through the center of mass $\mathscr C(P)$. This proves our theorem for the corrugated pentagram case, at least when $m \le n$.
\end{proof}

\subsubsection{The case where $2 \le n < m$} \label{subsection:3.3}  

Now we discuss the the remaining case, namely when $2 \le n < m$. Note that we can no longer lift points by adding extra coordinates, so instead of extending to higher dimensions, we restrict certain points to lower dimensions. We still start from the collection of $n$ points $A_1, A_{m+1}, ..., A_{(n-2)m+1}$, each consisting of $n$ points in general position in $\R^m$. All points in this collection span a subspace of dimension $n$ inside $\R^m$, since we can start from $n$ points in $A_1$ and reach all the points in $A_2$ following the direction along the $x_1$-axis, and reach all the points in $A_3$ along the $x_2$ direction, and so forth. Namely, they all live in a copy of $\R^n \subset \R^m$. Though the points in $A_i$ no longer span a hyperplane in $\R^m$, they certainly do in this copy of $\R^n$.

Now nothing else is changed.  We have $J_i = |A_i|$, which is an $(n-1)$ dimensional flat in $\R^m$ and a hyperplane in $\R^n$, and we have $n-1$ copies of these $(n-1)$-flats, in the same copy of $\R^n$ spanned by all the points considered. The transverse intersections are given for free this time, since the $A_i$ can be chosen such that the $(n-1)$ flats are in general position.  We then carry out the same analysis as in the previous case, and the desired result follows immediately. This together with the case $n \ge m$ concludes the proof of Theorem \ref{T003}.

\section{{Lower pentagram map and the mirror pentagram map}}\label{sec:lowerpentagram}

In this section, we introduce the notion of the lower pentagram map, define a new system called the mirror pentagram map, and prove the center of mass results of ``axis-aligned polygons'' for the mirror pentagram maps (Theorem \ref{T007}), for the lower pentagram maps (Theorem \ref{T008}) and for the cross ratio frieze patterns discussed in the introduction (Theorem \ref{T005}).
We point out that the methods to prove the seemingly algebraic results mentioned above, in particular Theorem \ref{T008}, are purely geometric. It is in fact similar to the proof of Theorem \ref{T002} for the usual pentagram map. 

This section is organized as follows: 
\begin{itemize}
\item We first introduce the definition of the lower pentagram maps. Based on the obstruction to taking diagonal polygons, we define the notion of the mirror pentagram map, which serves as taking $1$-diagonals of polygons in $\P^1$.
\item We discuss properties of the mirror pentagram map, and then move on to prove Theorem \ref{T007}, which is an analogous ``center of mass'' result on for the mirror pentagram map. Not surprisingly, this is achieved through a similar controlled lifting method. 
\item Based on this, we prove Theorem \ref{T008}. More specifically, given $A_0$ and $A_1$ as in  \ref{T005}, we lift the $n$ points in $A_1$ from $\P^1$ to $\P^2$. On these $n$ points in $\P^2$, we apply the mirror pentagram maps. We then show that Theorem \ref{T007} implies Theorem \ref{T008} by taking appropriate projections. 
\item In the end, we show that Theorem \ref{T008} implies Theorem \ref{T005} by restricting to some sub-rows of points in the cross ratio frieze pattern in \ref{T005}.
\end{itemize} 

\subsection{The lower pentagram map} 
\begin{figure}
\begin{center}
\includegraphics[scale=0.6]{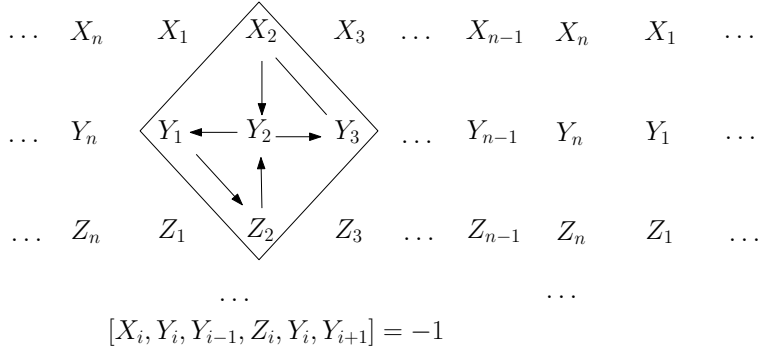} 
\end{center} 
\caption{\label{figure:lower_pentagram_definition}The definition of the lower pentagram map.}
\end{figure}

Gekhtman, Shapiro, Tabachnikov and Vainshtein considered a lower dimensional analog of the pentagram map  in \cite{GSTV}. Let $X, Y$ be two polygons (cyclically labelled points) in $\P^1$ in general position. The lower pentagram map $T_1$ takes the pair $(X, Y)$ to another pair of polygons $(Y,Z)$, where $Z = Z_1 Z_2 .... Z_m$ is determined from $X = X_1 X_2 ... X_m$ and $Y = Y_1 Y_2 ... Y_m$  by the following procedure: $Z_i$ is the unique point in $\P^1$ such that the six-point cross ratio 
$$[X_i, Y_i, Y_{i-1}, Z_{i}, Y_i, Y_{i+1}] = -1.$$ Here $[a, b, c, d, e, f] = \displaystyle \frac{(a-b)(c-d)(e-f)}{(b-c)(d- e)(f-a)}$. 

\begin{illustration}
An illustration of the definition is given in Figure \ref{figure:lower_pentagram_definition}
\end{illustration}

$T_1$ is analogous to the pentagram map $T$ in the following sense: let us consider a polygon $P \subset \P^2$ and apply the pentagram map to obtain $T(P)$ and $T^2 (P)$, as shown in Figure \ref{figure:lower_pentagram_and_pentagram}. 
Consider the 6 points $P_{i} \in P$ and $Q_{i \pm 1}, Q_{i \pm 3} \in T(P)$ and $R_i \in T^2 (P)$ in the figure. The Menelaus theorem from projective geometry says that the following ratio of signed lengths: 
$$\frac{P_i Q_{i-1}}{Q_{i-1} Q_{i-3}} \cdot \frac{Q_{i-3}R_{i}}{R_i Q_{i+1}} \cdot \frac{Q_{i+1} Q_{i+3}}{Q_{i+3} P_{i}} = -1. $$ With a moderate abuse of notation, we can write this as 
$$[P_{i}, Q_{i-1}, Q_{i-3}, R_{i}, Q_{i+1}, Q_{i+3}] = -1,$$ which is analogous to the $6$-point cross ratio that defines the $T_1$ map. Indeed, we know that if we project the $6$-points in Figure \ref{figure:lower_pentagram_and_pentagram} 
to any $\P^1 \subset \P^2$, the cross ratio will be preserved. In particular, we can take the projection in the direction $Q_{i+1} Q_{i-1}$ as suggested in the figure, in which case the projected points (after properly relabelling) satisfy the defining relation of the $T_1$ map: 
$$[X_i, Y_i, Y_{i-1}, Z_{i}, Y_i, Y_{i+1}] = -1.$$

\begin{figure}
\begin{center}
\includegraphics[scale=0.6]{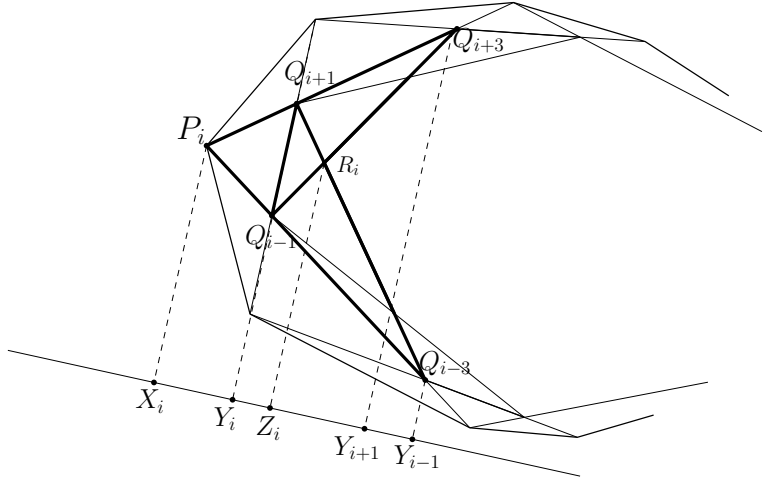} 
\end{center} 
\caption{\label{figure:lower_pentagram_and_pentagram} Analogy between the lower pentagram map and the pentagram map.} 
\end{figure}

In \cite{Glick1} Glick showed that certain sets of points under the lower pentagram map $T_1$ have the Devron (collapsing) property.  However, there was no concrete geometric description of such systems on $\P^1$. In the next subsection, we give a geometric construction of the lower pentagram map, in the same spirit of how the pentagram map was initially defined by Schwartz. 

With the supporting geometry, we can extend our theorems on pentagram maps and corrugated pentagram maps to the $1$-dimensional case.  First, we have an analogous notion of axis-aligned polygons in $\P^1$: 

\begin{definition}
Let $(X, Y)$ be described as above, we say the $2n$ points form an axis-aligned polygon if  $X$ only consists of  $n$ points at infinity. 
\end{definition}

Now, we can state the center of mass theorem for the lower pentagram map $T_1$, 

\begin{theorem}\label{T008}
Let $(\infty, B)$ be a generic axis-aligned $n$-gon  such that $T_1^{n-1}$ is defined, then $T_{1}^{n-1}(\infty, B) = (C, D)$ where $D = (d_1, d_2, ..., d_n)$ and $$d_1 = d_2 =... = d_n = \mathscr C (B).$$
\end{theorem}

\begin{example}
For example, if we take $3$ points $1, 2, 6$ as $B$, then we obtain the following pattern where $2$ iterations of $T_1$ give us the constant row with constant $3$.  \\
\begin{center}
\includegraphics[scale=0.5]{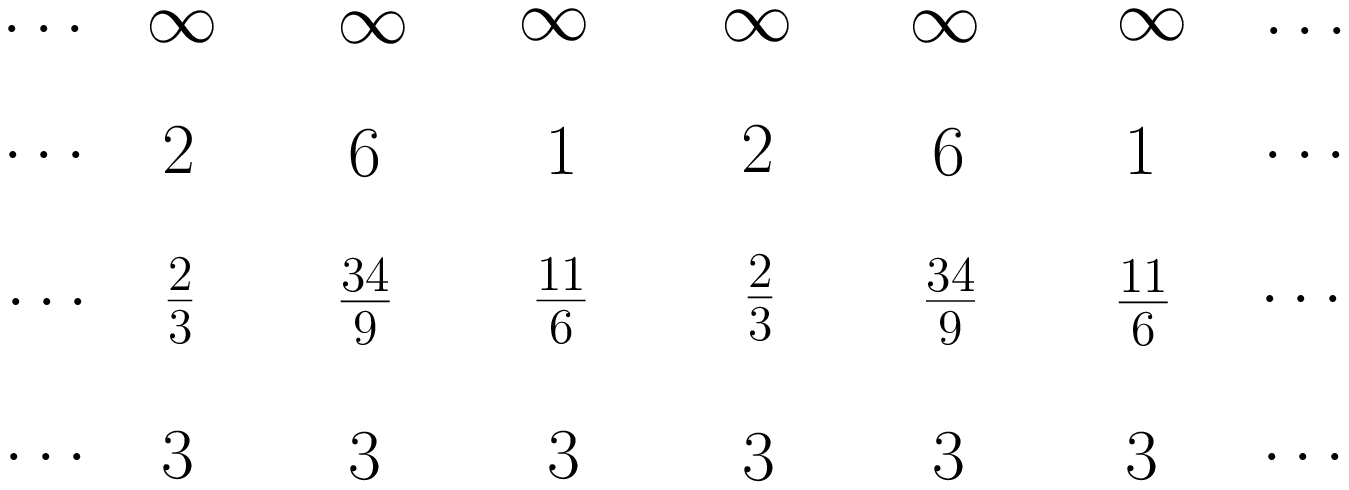} 
\end{center} 
\end{example}
The readers may notice that the rows in the example above are obtained by taking the $1^{st}, 3^{rd}, 5^{th}$ and $7^{th}$ rows from the example displayed at the end of the introduction. This observation will indeed provide a way to prove \ref{T005}, as we shall see near the end of the paper.   \\

\subsection{The mirror pentagram map} 
Now we proceed to discuss the mirror pentagram map in this subsection.  We first define a simple reflection map: 
\begin{definition}[The reflection map]
First we fix the projective line $l_0 \in  \P^2$ where the affine part of $l_0$ is the $x$-axis in the affine plane in $\P^2$. We define the reflection map $r: \P^2 \rightarrow \P^2$ by sending a point $X \in \P^2$ to its mirror image $X'$ about $l_0$. In particular, for points in the affine plane $\R^2$, $r((x, y)) = (x, -y)$.  (In homogeneous coordinates $r ((X, Y, Z)) = r (X, Y, -Z)$). It is clear that $r^2$ is the identity on $\P^2$. 
\end{definition}

We define the mirror pentagram (MP) map as the following: 
\begin{figure}
\begin{center}
\includegraphics[scale=0.35]{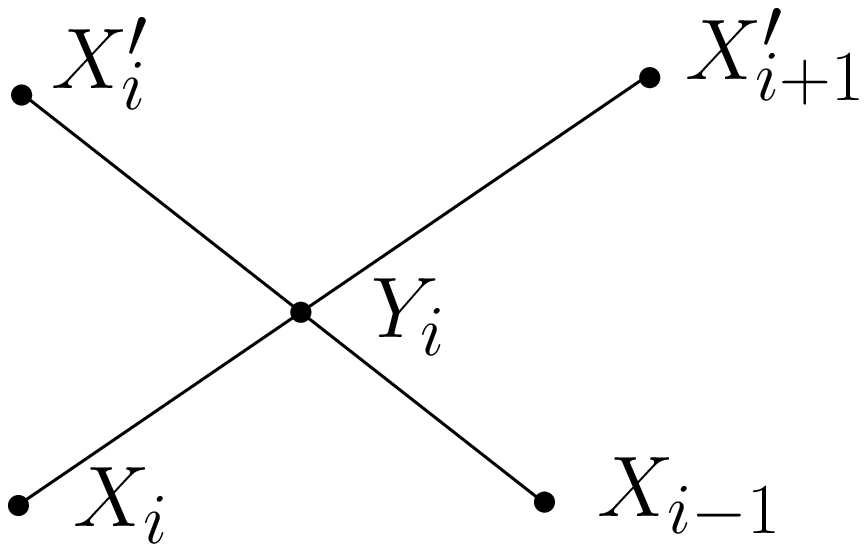}  \quad  \qquad \qquad
\includegraphics[scale=0.35]{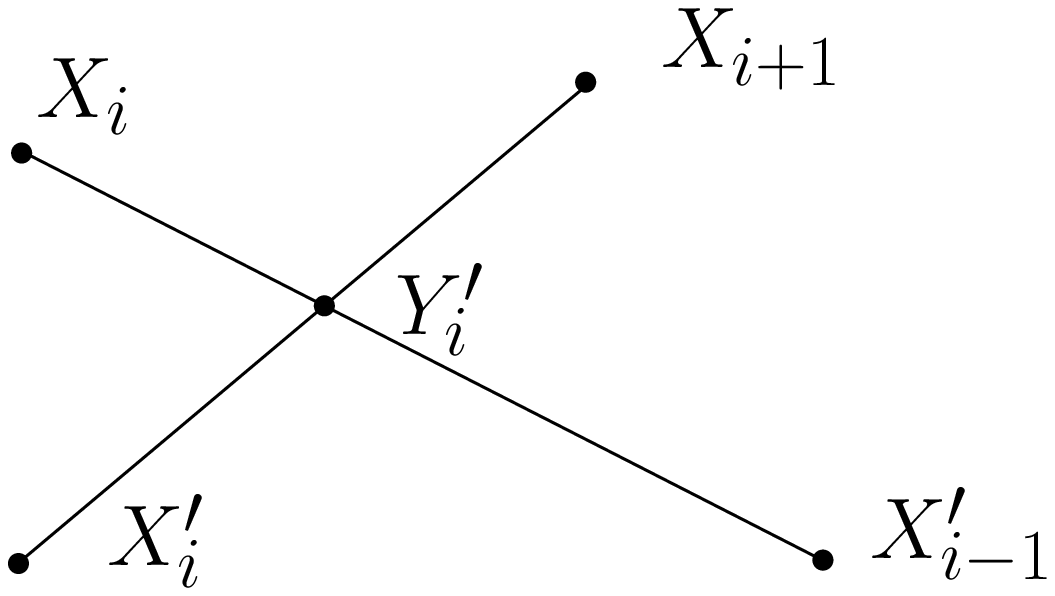}  
\end{center} 
\caption{\label{figure:mirror_pentagram_definition} The Mirror pentagram map.}
\end{figure}

\begin{definition}[The mirror pentagram map]
Consider a linearly ordered sequence of  $n$ points in general position  $P= X_1, X_2, ..., X_n \in \P^2$, let $$r (P) = \left\{r (X_1), r (X_1), ..., r (X_n)\right\}$$ be the linearly ordered sequence of the reflected points (the order being inherited from $A_1$). 
The Mirror pentagram map $MP$ sends $(P, r(P))$ to $(Q, r(Q))$, where $Q = Y_1, Y_2, ..., Y_n \in \R^2 $ is determined by the following rules:  we take four points $X_i, X_i', X_{i-1}$ and $X_{i+1}'$, and define the intersection of  $X_i X_{i+1}'$  and $X_{i-1}X_i'$ to be $Y_i$. In other words, $$Y_i = X_i X_{i+1}' \cap X_{i-1}X_i' .$$ Similarly, we define $$Y_i' = X_i' X_{i+1} \cap X_{i-1}' X_i. $$
\end{definition}

The definition of the pentagram map is illustrated in Figure \ref{figure:mirror_pentagram_definition}.

\begin{example}
The mirror map is best described through an example. In Figure \ref{figure:MP_map_example}, we start with $4$ points $X_1, X_2, X_3, X_4$, first apply the reflection map to get $X_1', X_2', X_3', X_4'$ and then apply the mirror pentagram map to obtain the $Y_i$ and $Y_i'$. \\
\end{example}

\begin{figure}
\begin{center}
\includegraphics[scale=0.6]{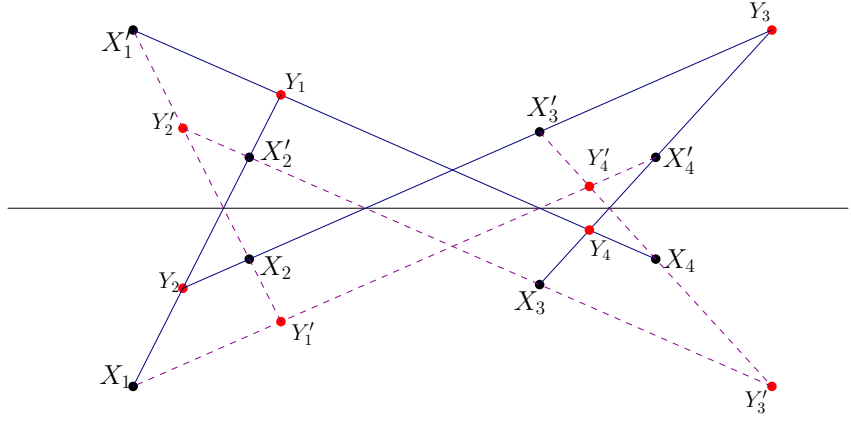} 
\end{center} 
\caption{\label{figure:MP_map_example}An example of the mirror pentagram map.} 
\end{figure}

The mirror pentagram map is well defined, since $$r(Y_i) = r( X_i X_{i+1}' \cap X_{i-1}X_i' ) =  r(X_i) r(X_{i+1}') \cap r(X_{i-1}) r(X_i') = Y_i'. $$

We remark that the mirror pentagram map commutes with the reflection $r$, in other words, for $(P, r(P))$, we have $$MP \circ r (P, r(P)) = MP (r(P), P) = r \circ MP (P, r(P)).$$ 


\begin{lemma}\label{lemma:mirror_pentagram_inverse} The MP map has an inverse: given $(Q, r(Q))$ in general position, we take the intersection of $Y_i Y_{i+1}$ and $Y_{i-1}' Y_{i}'$ to obtain $X_i$, and similarly  $Y_i' Y_{i+1}' \cap Y_{i-1} Y_{i} = X_i'$. 
\end{lemma}

\begin{example}
The definition of the inverse of the mirror pentagram map is illustrated in Figure \ref{figure:mirror_pentagram_inverse}. 
\end{example}

\begin{proof}
The fact that these maps are inverses to each other is easy to see, we show one direction and leave the other to the reader. Consider $Y_i Y_{i+1} \cap Y_{i-1}' Y_{i}'$, where each $Y_i$ is obtained from the $X_i, X_i'$ under MP, then $Y_i = X_i' X_{i-1} \cap X_{i+1}' X_i$ and $Y_{i+1} =  X_{i+1}' X_i \cap  X_{i+2}' X_{i+1}$, thus the line $Y_i Y_{i+1}$ is same as the line $ X_{i+1}' X_i$ by the assumption that no two lines are degenerate. Similarly, $Y_{i-1}' Y_{i}' = r (Y_{i-1} Y_{i}) = r(X_{i}' X_{i-1}) = X_{i} X_{i-1}'$, so $Y_i Y_{i+1} \cap Y_{i-1}' Y_{i}' = X_{i+1}' X_i \cap  X_{i} X_{i-1}' = X_i$. This proves that $MP^{-1} \circ MP$ is identity. The other direction is identical. 
\end{proof}

\begin{figure}
\begin{center}
\includegraphics[scale=0.35]{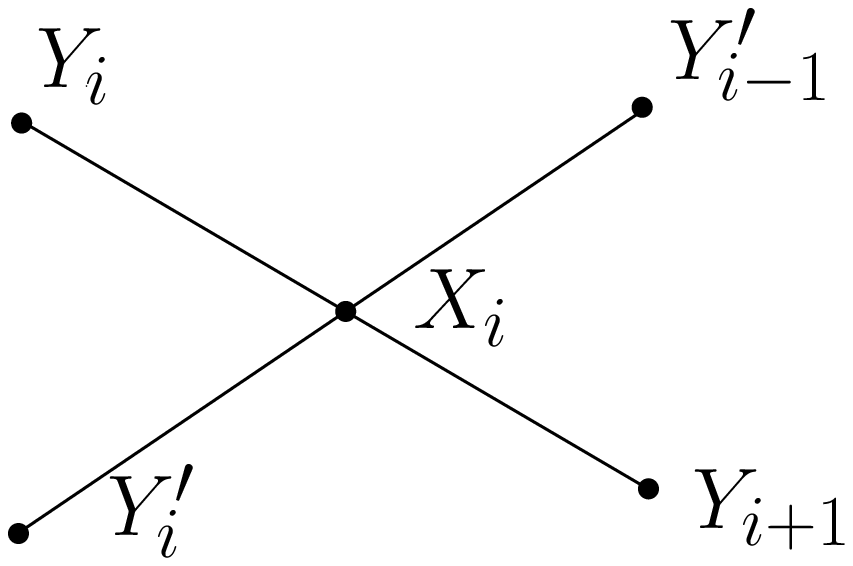}  \quad \qquad \qquad
\includegraphics[scale=0.35]{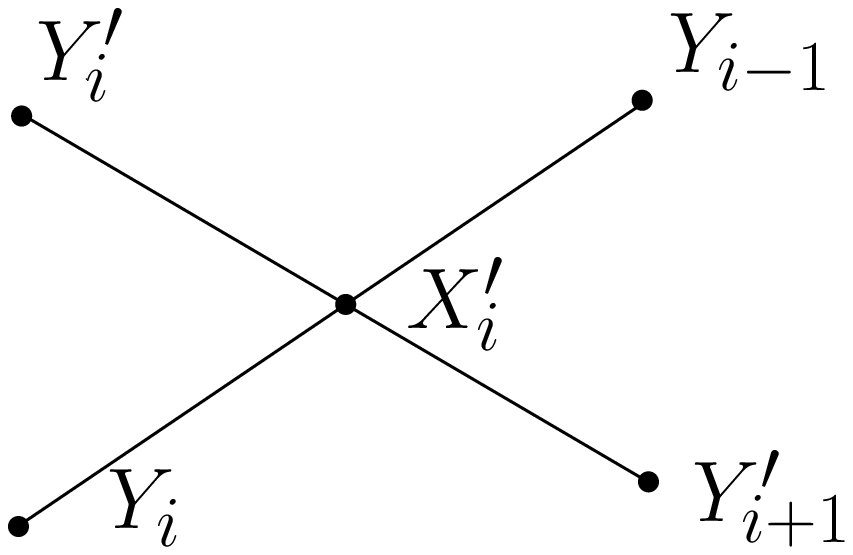} 
\end{center} 
\caption{\label{figure:mirror_pentagram_inverse} The inverse of the Mirror pentagram map.}
\end{figure}

The mirror pentagram map is closely related to the lower pentagram map, in fact, it can be viewed as a geometric definition of the lower pentagram map. Given some $P \subset \P^2$ or equivalently $(P, r(P)) $ in general position, let $(P^{-1}, r(P^{-1})) = MP^{-1}(P, r(P)) $ be the uniquely determined points by the inverse. Let $p: \P^2 \rightarrow \P^1$  be the projection $(X, Y, Z) \mapsto (X, Y)$ in homogeneous coordinates. Take the projection of the two pairs of n-points $p(P^{-1}, P)$, it is clear from definition that $p (P^{-1}, P) = (p(P^{-1}), p(P)) $.  We claim that $$ T_1 [p(P^{-1}), p(P)] =\Big(p(P),  p[\text{MP} (P, r(P))]\Big).$$

This seemingly complicated expression asserts that, starting from $n$-points $P$ in $\R^2$, we can find a corresponding pair $(A, B) $, where $B = p(P), A = p(P^{-1})$, such that applying the lower pentagram map $T_1$ to the pair $(A, B)$ is the same as applying the mirror pentagram map to $(P, r(P))$ and then project down using $p$.  An illustration of the claim is given in Figure \ref{figure:mirror_pentagram_and_lower_pentagram}. 
In fact, this process can be repeated (as long as concerned points are in general position), so we have 

\begin{lemma} \label{lemma:lower_pentagram_commutes_with_mirror_pentagram}
Retain notations from above, for any $k$ iterations: 
$$ T_1^k [p(P^{-1}), p(P)] =\Big(p[\text{MP}^{k-1} (P, r(P))],  p[\text{MP}^k (P, r(P))]\Big).$$
\end{lemma}
\begin{proof}
It suffices to prove the lemma for $k=1$, which is a direct consequence of Menelaus Theorem, as shown in Figure \ref{figure:mirror_pentagram_and_lower_pentagram}.
\end{proof}
\vspace*{0.2cm}

\subsection{The point of collapse under the mirror pentagram map}

For the lower pentagram map $T_1$, there is an analogous notion of axis-aligned polygon. We say a pair of $n$-points $(A, B) $ where $A = (A_1, A_2, ..., A_n)$ is axis-aligned if $A_1 =  A_2 = ... = A_n$. Suppose that $B = (B_1, B_2, ..., B_n)$ and that $B_i \ne A_1$ for any $i$, then we can define the center of mass of $\mathscr C (A, B)$ as follows:  first we apply a projective transformation $\varphi$ to $\P^1$ that sends $A_j$ to the point at $\infty$, by assumption $\varphi (B)$ sits in the affine line $\R^1 \subset \P^1$, define $C= \displaystyle \frac{1}{n}  (\varphi (B)_1 + \varphi (B)_2 + ... + \varphi (B)_n)$, and then define $$\mathscr C (A, B) = \varphi^{-1} (C) = \varphi^{-1} \big( \frac{1}{n} (\sum  \varphi(B)_{j}) \big). $$ 

\begin{remark}
Similar to the definition of axis-aligned polygon in Section \ref{sec:introduction}, $\mathscr C (A, B)$ is well defined, in other words, it does not depend on choices of $\varphi$ such that $\varphi (A_j) = \infty$. 
\end{remark}

\begin{figure}
\begin{center}
\includegraphics[scale=1]{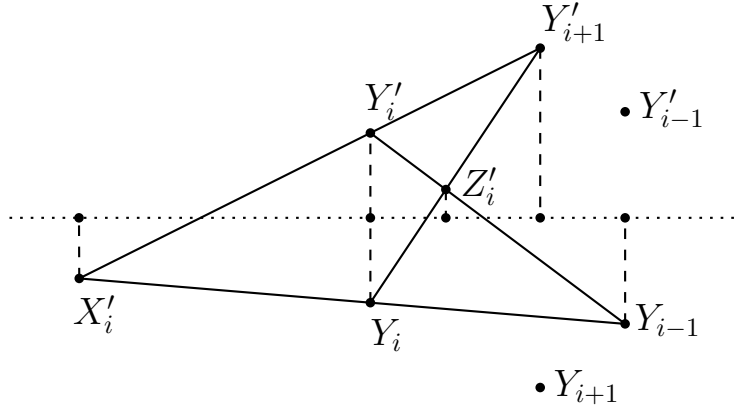} 
\end{center} 
\caption{\label{figure:mirror_pentagram_and_lower_pentagram} The correspondence between the Mirror pentagram map and the lower pentagram map.} 
\end{figure}

\begin{notation}
When $A$ is uniformly $\infty$, we denote $\mathscr C( A, B) = \mathscr C(B)$. 
Note that this corresponds to the case in the Mirror pentagram map where $P \subset \P^2$ lies on a horizontal line $l \ne l_0$, in other words, the affine part of $l$ is parallel but not equal to the $x$-axis, or equivalently, the intersection of $l$ and $l_0$ lies on the line at infinity. We call such a pair $(P, r(P)) $ an axis-aligned pair. For an illustration see Figure \ref{figure:mirror_pentagram_collapse}. 
\end{notation}


The main theorem that we will prove in this section is the following: 

\begin{theorem}{\label{T007}}  Let $(P, r(P))$ be a generic axis-aligned pair where $P$ consists of $n$ points on a line $l$ parallel to $l_0$. Suppose that the $(n-1)$ iteration of the mirror pentagram map $MP^{n-1}$ is well defined on $(P, r(P))$, then $MP^{n-1}$ takes $(P, r(P))$ to two points $(s, r(s)) \in \P^2 \times \P^2$, where $$p(s) =\mathscr C (p (P)).$$  
in other words,  for any such $(P, r(P))$, let $(S, r(S)) = \text{MP}^{n-1} (P, r(P))$, then $S_1 = S_2 = ... = S_n = s$, and $$s = \frac{1}{n} (p(P_1) + p(P_2) + ... + p(P_n)).$$
\end{theorem}

An example of theorem when $n = 3$ is again given by Figure \ref{figure:mirror_pentagram_collapse}. 

\begin{remark}
The theorem asserts that the $(n-1)$ iteration of the mirror pentagram map takes $(P, r(P))$ to two points $s, r(s)$ that are mirror images of each other. We will see in the proof that when $n$ is even, $s = r(s)$, in other words, $s= (\mathscr C (P), 0)$. While for the case when $n = 2k-1$ is odd, $s =(\mathscr C (P), -\frac{1}{n}) $. 
\end{remark}

Similar to Section \ref{sec:pentagram}, we construct liftings of certain polygons into $\R^n$ and show that they intersect appropriately.  For different parities of $n$, we need to construct slightly different liftings. In particular, our next lemma needs to be treated differently depending on whether $n$ is even or odd. 
\begin{figure}
\begin{center}
\includegraphics[scale=0.8]{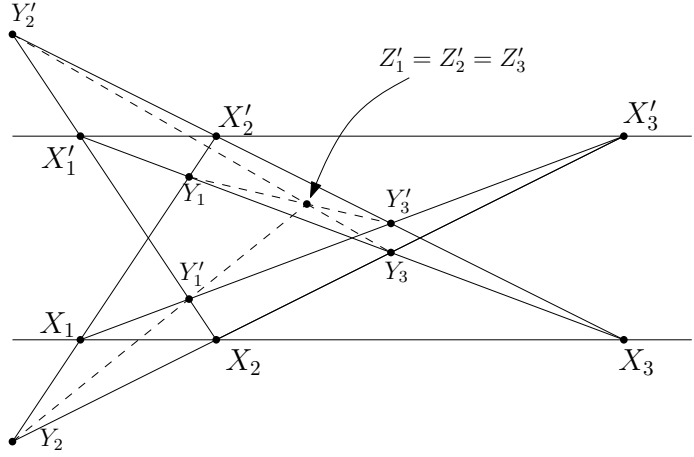} 
\end{center} 
\caption{\label{figure:mirror_pentagram_collapse} Point of collapse of axis-aligned under the mirror pentagram map.} 
\end{figure}

\subsubsection{$n$ is even.} \label{subsubsection:even}
First we assume that $n = 2k$ is even, the case when $n = 2k+1$ involves slightly different definitions, which we treat later in the section. \\

We identify $\R^2$ with the copy of affine plane in $\R^n$ where $x_3, ..., x_n$ coordinates are $0$ as in Section \ref{sec:pentagram}. Consider an axis-aligned pair $(P, r(P))$ where $P = \left\{ X_1, X_2, ..., X_{2k} \right\}$, and consider the following $(n-1)$ linearly ordered sequences of $n$ points: 

\begin{align*} A_1 & = \left\{X_1(1),\quad  X_2'(1), \quad  X_3(1), \quad..., \quad X_{2k-1}(1), \quad X_{2k}'(1)\right\}   \\ A_3 & = \left\{X_2(3),\quad  X_3'(3), \quad  X_4(3), \quad..., \quad X_{2k}(3), \quad X_{1}'(3)\right\}   \\ A_5 & = \left\{X_3(5),\quad  X_4'(5), \quad  X_5(5), \quad..., \quad X_{1}(5), \quad X_{2}'(5)\right\}  \\  & \ldots \\ A_{2n-3} &= \left\{X_{n-1}(2n-3),\quad  X_n'(2n-3), \quad ..., \quad X_{n-3}(2n-3), \quad X_{n-2}'(2n-3)\right\}  \end{align*} \\ again we think of $X_i (j), X_i (j+4), ...$ as different points that coincide in $\R^2$.  Without loss of generality, we may assume that the line $l \in \R^2$ where points in $P$ lie on is the line $y = -1$.  Then  all points in $A_j$ have coordinates $(x_1, \pm 1, 0, ..., 0)$ for some $x_1 \in \R$. 

\begin{notation}
We denote $A_i' = r (A_i)$ which consists of the mirror image of each point in $A_i$.  Let $W_1 = \left\{A_1, A_3, ..., A_{2n-3}\right\}$ and $W_1 ' = \left\{A_1', A_3', ..., A_{2n-3}'\right\} $.
\end{notation}


\begin{lemma}\label{lemma:mirror_pentagram_mating_n_even}

Let $n = 2k$ be an even integer with $k \ge 2$. Let $P$ be defined as in Theorem \ref{T007} and $W_1$ be defined as above, then for $j \in [2, n-1]$, the map $MP^{j-1}$ is defined on $(P, r(P))$ if and only if the mating process $$W_1 \rightarrow W_2 \rightarrow ... \rightarrow W_j$$ is defined. Suppose that $MP^{i-1}$ is well defined where $ i \le n-2$, then the union of all sets in $W_i$ gives the points in $MP^{i-1} (P, r(P)).$ If $MP^{n-2}$ is well defined, and suppose that $MP^{n-2}(P, r(P)) = (Q, r(Q))$ where $Q = (Q_1, Q_2, ..., Q_{2k})$, then $$W_{n-1} = (Q_{2k-1}, Q_{2k}', Q_1, Q_2' , ..., Q_{2k-3}, Q_{2k-2}').$$
 \end{lemma}

\begin{proof}
This is simply unfolding the definitions. Assume that $W_1 \rightarrow W_2$ is well defined, where $W_2 = (B_2, B_4, ..., B_{2n-4})$. Then $B_2 = A_1 * A_3 = \left\{Y_2(2), Y_3'(2), ..., Y_1'(2) \right\}$ from the definition of the mirror pentagram map, hence $MP$ is well defined on $(P, r(P))$ if and only if $A_1 * A_3$ is well defined, if and only if $A_{2s-1} * A_{2s+1}$ is well defined for all relevant indices $s$, since the pentagram map commutes with the reflection map $r$. Repeat the argument $j-1$ times and we get the first part of the lemma.  The second part of the lemma is trivial from the construction of $W_1$, and the third  part of the lemma comes from carefully keeping track of indices of each step of the mating process. 
\end{proof}

\subsubsection{$n$ is odd.}\label{subsubsection:odd}  Now we assume that $n = 2k-1$ is odd for $k \ge 2$. Let $P$ be defined as in Theorem \ref{T007} as usual. In this case we define

\begin{align*} A_1 & = \left\{X_1(1),\quad  X_2'(1), \quad  X_3(1), \quad..., \quad X_{2k-2}'(1), \quad X_{2k-1}(1)\right\}   \\ A_3 & = \left\{X_2(3),\quad  X_3'(3), \quad  X_4(3), \quad..., \quad X_{2k-1}'(3), \quad X_{1}(3)\right\}   \\ A_5 & = \left\{X_3(5),\quad  X_4'(5), \quad  X_5(5), \quad..., \quad X_{1}'(5), \qquad X_{2}(5)\right\}  \\  & \ldots \\ A_{2n-3} &= \left\{X_{n-1}(2n-3),\quad  X_n'(2n-3), \quad ..., \quad X_{n-3}'(2n-3), \quad X_{n-2}(2n-3)\right\} \\  A_{2n-1} &= \left\{X_{n}(2n-1),\quad  X_1'(2n-1), \quad ..., \quad X_{n-2}'(2n-1), \quad X_{n-1}(2n-1)\right\}  \end{align*}

Note that since the $n^{th}$ element and the $1^{st}$ element are both in $P$, when we take a successive pair, for example $A_{1}, A_3$, the mating process will produce an element $ \overline{X_{2k-1} X_1} \cap \overline{X_1 X_2}$ which is not defined, since $(P, r(P))$ is an axis-aligned pair.  Hence we need to redefine the mating process for the case when $n$ is odd. 

\begin{definition}[The star operation]
Let $ \alpha = (x_1, x_3, ..., x_{2n-1})$ and $\beta = (y_1, y_2, ...,$ $y_{2n-1})$ be linearly ordered sequences of $n$ points, we define the star operation on $\alpha$ and $\beta$ by constructing $\gamma = \alpha \star \beta$, where $\gamma = (z_2, z_4, ..., z_{2n-2})$ is a linearly ordered sequence of $(n-1)$ points defined by 
$$z_j = \overline{x_{j-1} x_{j+1}} \cap \overline{y_{j-1} y_{j+1}}.$$ 
\end{definition}

\begin{remark} Note that the difference between $\alpha * \beta$ and $\alpha \star \beta$ is simply that $\alpha \star \beta$ drops the last point from $\alpha * \beta$, hence contains one less point than $\alpha$ or $\beta$. 
\end{remark}

\begin{definition}[The star-mating process]
Suppose that each $\alpha_{1,1}, \alpha_{1,3}, ..., \alpha_{1, 2m-1}$ is a linearly ordered $n$ point sequence.  Let $S_1 = (\alpha_{1,1}, \alpha_{1,3}, ..., \alpha_{1, 2m-1})$, then we say $S_1 \leadsto S_2 $ if each star-mating $ \alpha_{1, 2i-1} \star \alpha_{1, 2i+1}$ is well defined and  $$S_2 = (\alpha_{2,2}, \alpha_{2,4}, ..., \alpha_{2, 2m-2}) \quad \text{where }  \alpha_{2, 2i} = \alpha_{1, 2i-1} \star \alpha_{1, 2i+1}.$$ 
The process $S_1 \leadsto S_2$ is called the star-mating process.
\end{definition}

\begin{remark}
For $S_1 \leadsto S_2 \leadsto ... \leadsto S_m$ to be well defined, we need to require $m \le n$. Note that if $S_1$ contains $n$ points, then $S_i$ contains $n-i+1$ points. 
\end{remark}
Now let \begin{align*} 
W_1(1) & =   \left\{A_1, A_3, ..., A_{2n-3}\right\} ;\\ 
W_1(2) & =   \left\{A_3, A_5, ..., A_{2n-1}\right\} ; \\
W_1(3) & =   \left\{A_5, A_7, ..., A_{1}\right\} ;\\ 
\ldots \\
W_1 (n) & =  \left\{A_{2n-1}, A_1, ..., A_{2n-5}\right\} . \end{align*} Each $W_1(j)$ is a linearly ordered sequence of $n$ points. 

\begin{lemma}
\label{lemma:mirror_pentagram_mating_n_odd}

Let $n = 2k-1$ with $k \ge 2$. Let $P$ be defined as in Theorem \ref{T007} and $W_1(l)$ be defined as above for $l = 1, ..., n$, then for $j \in [2, n-1]$, the map $MP^{j-1}$ is defined on $(P, r(P))$ if and only if the star-mating process $$W_1(l) \leadsto W_2(l)  \leadsto ...  \leadsto W_j(l)$$ is defined for all $l \in [1, n]$.  Suppose that $MP^{i-1}$ is well defined where $ i \ge n-1$, then the union of all sets in $W_i$ gives the points in $MP^{i-1} (P, r(P)).$ In particular, if $MP^{n-2}$ is well defined, and suppose that $MP^{n-2}(P, r(P)) = (Q, r(Q))$ where $Q = (Q_1, Q_2, ..., Q_{2k})$, then $$W_{n-1}(1) = (Q_{n-1}, Q_{n}'); \quad W_{n-1}(2) = (Q_{n}, Q_{1}'), \quad ... , \quad W_{n-1}(n) = (Q_{n-2}, Q_{n-1}')$$
 \end{lemma}

\begin{proof} The proof is again keeping track of definitions and relevant indices and analogous to the proof of Lemma \ref{lemma:mirror_pentagram_mating_n_even}, where the only difference is that now we consider a sequence of star-mating processes instead of one mating process.
\end{proof}

As in Section \ref{sec:pentagram}, we define a \emph{lifting} of the sequence $A_1, A_3, ..., A_{2n-3}$ to be a way to lift each $A_i$ into a joint in $\R^n$ while fixing the $x_1, x_2$ -coordinates for every point. Analogously, a \emph{parallel lifting} $L$ of the sequence of $n$-points $A_1, A_3, ..., A_{2n-3}$ is a family of liftings such that $L$ sends $A_1, ..., A_{2n-3}$ to $\widetilde A_1, ..., \widetilde A_{2n-3}$, where the $l^{th}$ point of each $\widetilde A_{i}$ has the same $x_3, x_4, ..., x_n$ coordinates, for all $l =1, 2, ..., n$. 

\begin{remark} By Lemma \ref{lemma:parallel_lifting_into_joint} we know that a parallel lifting lifts the sequence of $n$ points defined above $A_1, A_3, ..., A_{2n-3}$ to a polyjoint $(\widetilde A_1, \widetilde A_3, ..., \widetilde A_{2n-3} ).$
\end{remark}


\begin{lemma}{\label{lemma:mirror_pentagram_lifting_of_centroids}} Suppose that a polyjoint $(\widetilde A_1, \widetilde A_3, ..., \widetilde A_{2n-3} )$ is lifted from $A_1, A_3..., $ $A_{2n-3}$ by a parallel lifting. Let $ C_i \in \R^n$ be the centroid of the $n$-vertices in $\widetilde A_i$ for each $i = 1, 3, ..., 2n -3$.  Then
\begin{equation*}  C_1 = C_3 = ... = C_{2n-3};  
\end{equation*} and 
\begin{equation*}
\begin{cases}
 \pi (C_1) = ... = \pi (C_{2n-3}) = (\mathscr C (P), 0)  & \text{ if  $n$ is even}; \\ 
  \pi (C_1) = ... = \pi (C_{2n-3}) = (\mathscr C (P), \frac{-1}{n})   & \text{ if  $n$ is odd}.
 \end{cases}
 \end{equation*}
\end{lemma}
\begin{proof}
The proof of first part is analogous to the proof of Lemma \ref{lemma:parallel_lifting_of_centroids}. For the second assertion, we note that for each $A_i$, when $n = 2k$, there are $k$ points in $A_i$ that lie on $l$ with $x_2$-coordinates $-1$ and $k$ points on $r(l)$ with $y$-coordinates $1$; while when $n = 2k-1$, there are $k$ points on $l$ and $k-1$ points on $r(l)$, so the $x_2$-coordinate of the average of all vertices is $-1/n$.
\end{proof}


\begin{lemma}{\label{lemma:mirror_pentagram_lifting_into_general_positions}} There exists an axis-aligned polygon with a parallel lifting of $A_1, A_3 ..., $ $A_{2n-3}$ such that all subspaces $|\widetilde A_1|, |\widetilde A_3|, ..., |\widetilde A_{2n-3}|$ are hyperplanes in general position. 
\end{lemma}

\begin{proof}
The proof is analogous to the proof of Lemma \ref{lemma:lifting_into_general_positions}, for which we set the $x$-coordinates of the points to be $0, 1, 2, ..., n-1$ and let the $y$-coordinates vary. Here we have to require the $y$ coordinates of the sequence of points to be $1, -1, 1, -1, ... $, but we can let $x$ vary. Other than the change of indices all other calculations remain the same, namely we can choose 1 $x$-coordinate to be large and all other $x$-coordinates to be sufficiently small such that the determinant of the matrix is nonzero. Hence by a similar argument of Lemma \ref{lemma:lifting_into_general_positions}, the statement of the lemma follows.  
\end{proof}

For a parallel lifting, we label the prisms by $$T_2 = \widetilde A_1 \widetilde A_3, \quad T_4 = \widetilde A_3 \widetilde A_5, \quad  ..., \quad T_{2n-4} = \widetilde A_{2n-5} \widetilde A_{2n-3}.$$ Since $(P, r(P))$ is an axis-aligned pair, it is clear that all the prisms coincide, in other words, $$T_2 = T_4 = ... = T_{2n-4} := T.$$


Recall that a parallel lifting of $A_1, ..., A_{2n-3}$ is good if it satisfies the requirements in Lemma \ref{lemma:mirror_pentagram_lifting_into_general_positions}, and a good lifting is perfect if the sequence $(\widetilde A_1, \widetilde A_3, ..., \widetilde A_{2n-3})$ is fully sliced. 

\begin{lemma}{\label{lemma:mirror_pentagram_perfect_lifting}} For a generic axis-aligned pair $(P, r(P))$, let $A_1, ..., A_{2n-3}$ be as constructed, then there exists a perfect lifting into $\R^n$.  \end{lemma}

\begin{proof}
This follows directly from lemma 3.4 in Schwartz \cite{Schwartz2}.
\end{proof}

Next we prove the main theorem of the section.

\begin{proof}[Proof of Theorem \ref{T007}] \indent

1. $n$ is even.

We know that $(P, r(P))$ is a generic axis-aligned polygon with a perfect lifting into $\R^n$. Let $(A_1, A_3, ..., A_{2n-3})$ and $W_1$ be defined as in \ref{subsubsection:even}. By assumption we can lift $W_1 = (A_1, ..., A_{2n-3})$ into $(\widetilde A_1, \widetilde A_3, ..., \widetilde A_{2n-3})$ through a perfect lifting. Let $T$ be as defined above, and we know that the polyjoint $(\widetilde A_1, \widetilde A_3, ..., \widetilde A_{2n-3})$ is fully sliced.  By assumption, $MP^{n-2}$ is well defined on $(P, r(P))$, so according to Lemma \ref{lemma:mirror_pentagram_mating_n_even}, $(n-2)$ iterations of the mating process is well defined on $W_1$. We have:
$$W_1 \rightarrow W_2 \rightarrow ... \rightarrow W_{n-1}.$$ Now define $H_{g,k}$ similarly as in Section \ref{sec:pentagram} where $H_{1,i} = \widetilde A_i$, and also define $Z_{g,k} = \pi (H_{g, k} \cap \Sigma_g T)$ for all relevant indices [note that in this case $T = T_2 = ... = T_{2n-4}$]. Let $Z_g = (Z_{g, g}, Z_{g,g+2}, ..., Z_{g, 2(n-1)-g})$, so $Z_1 = W_1$. Apply Lemma \ref{lemma:fully_sliced_mating_process} to the fully sliced polyjoint  $(\widetilde A_1, \widetilde A_3, ..., \widetilde A_{2n-3})$, and we get: 
$$Z_1 \rightarrow Z_2 \rightarrow ... \rightarrow Z_{n-1}.$$ Thus $W_i = Z_i$ for all relevant $i$. Now we repeat the argument in the proof of Theorem \ref{T002} using Lemma \ref{lemma:mirror_pentagram_lifting_of_centroids} and Lemma \ref{lemma:mirror_pentagram_lifting_into_general_positions}, and obtain that the line $l_{n-1} = \pi (H_{n-1, n-1})$ goes through all the points in $Z_{n-1}$, and also goes through the point $$\pi (C_1) = \pi(C_2) = ... \pi(C_{2n-3}) = (\mathscr C(P), 0).$$ Let $MP^{n-2} (P, r(P)) = (Q, r(Q))$ and $Q = (Q_1, Q_2, ..., Q_n)$, then since $$Z_{n-1} = W_{n-1} = (Q_{n-1}, Q_n', Q_1, Q_2', ..., Q_{n-3}, Q_{n-2}'), $$ all these points in $W_1$ goes through the line $l_{n-1}$. By symmetry we know that there is a line $l_{n-1}'$ that goes through all the points in $$W_n' = r(W_n) = (Q_{n-1}', Q_{n}, Q_1' Q_2, ..., Q_{n-2}),$$ and also goes through $r(\mathscr C (P), 0) = (\mathscr C(P), 0)$.  Since $MP$ is well defined on $(Q, r(Q))$ by assumption, we know that $l_{n-1}$ and $l_{n-1}'$ are distinct lines so intersect at most $1$ point. So their common point $(\mathscr C(P), 0)$ is the intersection $l_{n-1} \cap l_{n-1}' $. By the definition of the mirror pentagram map, we know that $(\mathscr C(P), 0)$ is precisely the degenerate point of $MP (Q, r(Q)) = MP^{n-1}(P, r(P))$. Hence the point $s$ in the statement of the theorem is $s = (\mathscr C(P), 0)$, where $p(s) = \mathscr C(P)$. 

2. $n$ is odd. 

This is a slight variant of the proof above. This time we let $W_1(1) = (A_1, A_3, ...,$  $A_{2n-3})$ and through a perfect lifting obtain $\widetilde W_1(1) = (\widetilde A_1, ..., \widetilde A_{2n-3})$, which is by definition fully sliced. Since $MP^{n-2}$ is well defined on $(P, r(P))$, by Lemma \ref{lemma:mirror_pentagram_mating_n_odd} we know that the following star-mating is well defined:
 $$W_1(1) \leadsto W_2(1)  \leadsto ...  \leadsto W_{n-1}(1).$$ Analogous to the previous case, we define $H_{1,i}$ and $H_{g,k}$, and let $Z_{g, k} = \pi (H_{g,k} \cap \Sigma_g T). $ Note that Lemma \ref{lemma:fully_sliced_mating_process} no longer applies since the mating process on any two pairs $Z_{g,k}$ and $Z_{g,k+2}$ is not well defined. In order to use the lemma we make the following definition: for each $Z_{g,k}$, an ordered sequence of $n$ points, we take the first $n-g+1$ points from $Z_{g,k}$ to form a new sequence $V_{g,k}$, and let $V_{g} = (V_{g,g}, V_{g, g+2}, ..., V_{g, 2(n-1)-g}).$ Lemma \ref{lemma:fully_sliced_mating_process} implies that, if star-mating process is well defined on all $V_i$, then $$V_1 \leadsto V_1 \leadsto  ... \leadsto V_{n-1}$$ since $V_1 = Z_1 = W_1(1)$ and $\widetilde W_1 (1)$ is fully sliced.  Therefore, we know that $V_{n-1}= W_{n-1}(1) $. Let $MP^{n-2}(P, r(P)) = (Q, r(Q))$, then $$V_{n-1} = W_{n-1}(1) = (Q_{n-1}, Q_n')$$ by Lemma \ref{lemma:mirror_pentagram_mating_n_odd}. By a similar argument as in the even case, we know that the line $l_{n-1}(1) = Q_{n-1} Q_n'$ goes through the point $(\mathscr C(P), -1/n)$. By symmetry we know that, all lines $$l_{n-1}(2) = Q_n Q_1',\quad l_{n-1}(3) = Q_1 Q_2', ...,\quad l_{n-1} (n) = Q_{n-2} Q_{n-1}'$$ go through the point $(\mathscr C(P), -1/n)$. By assumption $MP^{n-1}$ is well defined on $(P, r(P))$, so the lines are pairwise distinct. So the lines are concurrent and intersect at the common point $(\mathscr C(P), -1/n)$. Let $MP^{n-1}(P, r(P)) = (S, r(S))$ as in the statement of the theorem, by the definition of the mirror pentagram map, $$S_i  = Q_{i-1}Q_i' \cap Q_i Q_{i+1}' = l_{n-1}(i+1) \cap l_{n-2} (i+2),$$ which implies $S_1 = S_2 = ... S_n = (\mathscr C(P), -1/n)$.  

This proves the theorem.
\end{proof}
\vspace*{0.2cm}

\subsection{Some interesting corollaries}
As a corollary, we obtain Theorem \ref{T005}, which answers one of the conjectures of Glick's.  Before we prove the theorem, we first observe that whenever the lower pentagram map is well defined, the corresponding mirror pentagram map is also well defined. We state the observation in the following lemma:

\begin{lemma}\label{lemma:lower_map_to_mirror_map}
\label{lemma:lower_pentagram_well_defined}
Let $(\infty, B) $ denote an axis-aligned polygon $(A, B) $ where $A = (\infty, \infty, ..., \infty)$. Let $(P, r(P))$ be an axis-aligned pair such that $p(P) = B$,  then $MP^k$ is well defined on $(P, r(P))$ if $T_1^k$ is well defined on $(\infty, B )$. 
\end{lemma}

\begin{theorem*}[\ref{T008}]
Let $(\infty, B)$ be a generic axis-aligned $n$-gon such that $T_1^{n-1}$ is defined, then $T_{1}^{n-1}(\infty, B) = (C, D)$ where $D = (d_1, d_2, ..., d_n)$ and $$d_1 = d_2 =... = d_n = \mathscr C (B).$$
\end{theorem*}

\begin{proof}
We lift the polygon $B \subset  \R^1 \subset \P^1$ to a polygon $P$ on a horizontal line $l$, which is parallel to the $x$-axis, so $p (P) = B$, and $P^{-1} = (\infty)$, where $(P^{-1}, r(P^{-1})) = MP^{-1} (P, r(P))$.
Lemma \ref{lemma:lower_map_to_mirror_map} asserts that $MP^{n-1}$ is well defined on $(P, r(P))$. By Theorem \ref{T007}, we know that $MP^{n-1} (P, r(P))$ consists of two points $(s, r(s))$, where $s = \mathscr C(B)$. Hence $$p[MP^{n-1} (P, r(P))] = \mathscr C(B).$$
Now, by Lemma \ref{lemma:lower_pentagram_commutes_with_mirror_pentagram},
\begin{eqnarray*} T_1^{n-1}(\infty, B)  & = & T_1^{n-1} [p(P^{-1}), p(P)] \\ & = &  \Big(p[\text{MP}^{n-2} (P, r(P))],  p[\text{MP}^{n-1} (P, r(P))]\Big) \\ &=& \Big(C, \mathscr C(B) \Big). 
\end{eqnarray*}
In the equation, we have $T_{1}^{n-1}(\infty, B) = (C, D)$ where $C = p[\text{MP}^{n-2} (P, r(P))]$ is some $n$-gon in $\P^1$, and $D = \mathscr C(B)$ as desired.
\end{proof}

\begin{remark}
This strengthens theorem $6.15$ in Glick's paper \cite{Glick1} for the case when the polygon he considers is closed.  Glick showed that for a twisted polygon, under some iterations of the twisted pentagram map, the twisted polygon collapses to a point. We predict the position and number of iterations for the case of closed polygons.
\end{remark}

Now we return to Theorem \ref{T005} stated in the introduction. Recall that we consider the following chessboard patterns of integers in the plane: 

\begin{equation*}
\arraycolsep=0.05pt
\medmuskip = 2mu
\begin{array}{ccccccccccccccccccccccc}
... \hspace*{0.2cm}&  & X_{(0, 2n)}  & & X_{(0,2)}  & & X_{(0,4)}  & & ...  & &  X_{(0,2n)} & & X_{(0, 2)} & &  \hspace*{0.2cm} ...  \\
&...&  & X_{(1, 1)} & & X_{(1, 3)} & & ... \hspace*{0.1cm}& &\hspace*{0.1cm} ... && X_{(1, 1)} && ...  \\
...\hspace*{0.2cm} &  & X_{(2, 2n)}  & & X_{(2,2)}  & & X_{(2,4)}  & & ...  & &  X_{(2,2n)} & & X_{(2, 2)} & &   \hspace*{0.2cm} ... \vspace*{0.2cm} \\
\vspace*{0.05cm}
& & && && \vdots  \\
& ... & & X_{(2n-1, 1)} & & X_{(2n-1, 3)} & & ...  \hspace*{0.1cm}& &\hspace*{0.1cm}... && X_{(2n-1, 1)} && ... \\
...\hspace*{0.2cm} &  & X_{(2n, 2n)}  & & X_{(2n,2)}  & & X_{(2n,4)}  & & ...  & &  X_{(2n,2n)} & & X_{(2n, 2)} & &   \hspace*{0.2cm} ...  
\end{array}
\end{equation*} \vspace*{0.2cm} \\ 
where $A_{2i} = \{X_{(2i, 0)}, X_{(2i, 2)}, ..., X_{(2i, 2n)}\}$ and $A_{2i+1}= \{X_{(2i+1, 1)}, X_{(2i+1, 3)}, ..., $ $ X_{(2i+1, 2n-1)}\}$ for $0 \le i \le n-1$, and the corresponding cross ratios are all $-1$ as described in Section \ref{sec:introduction}. 

\begin{lemma}\label{lemma:grid_embedding}
Let $A_0, A_1, ..., A_{2n-1}$ be as above, then for any pair $A_{i-2}, A_{i}$, $2 \le i \le 2n-3$, the lower pentagram map $T_1$ is well defined on $(A_{i-2}, A_{i})$ and $$T_1(A_{i-2}, A_{i}) = (A_{i}, A_{i+2}).$$ 
\end{lemma}

\begin{proof}

We want to show that for any sub-pattern as the following 
\begin{equation*}
\arraycolsep=0.5pt
\medmuskip = 4mu
\begin{array}{ccccccccccccccccccccccc}
 & & &  & & V_2 & & & & &  \\
& & & & X_1 & &X_3 & & & & &&  \\
& & & Y_0& & Y_2& &  Y_4 & & & &  \\
& & & & Z_1 & & Z_3 & &&  \\
&&&&& W_2
\end{array}
\end{equation*} 
where the cross ratios $$[V_2, X_1, Y_2, X_3] = [ X_1, Y_0, Z_1, Y_2]  =[X_3, Y_2, Z_3, Y_4]  = [Y_2, Z_1, W_2, Z_3] =  -1,$$ then $$[V_2, Y_2, Y_0, W_2, Y_2, Y_4] = -1.$$

It suffices to consider the five points $V_2, X_1, X_3, Y_0, Y_4 \in \P^1$, since the rest of the points are uniquely determined given that the cross ratios are well defined.  Note that the cross ratios are preserved under projective transformations in $\P^1$, so we may assume that $V_2 = \infty$, thus $Y_2 = \frac{X_1+X_3}{2}$. 

Now it is just a matter of computation to solve for $W_2$ from the $4$ 4-term cross ratios, where we obtain that $$W_2 = \frac{(X_1+X_2)^2 - 4 Y_0 Y_4}{4 (X_1 + X_2 - Y_0 - Y_4)}.$$ It is easy to verify that $W_2$ satisfies the $6$-term cross ratio  as desired. 
\end{proof}

\begin{remark}
There is also a slightly more complicated, but also more elegant geometric proof of the lemma. 
\end{remark}

\begin{lemma}
Let $X_1, X_3, X_5 \in \P^1$ be $3$ points. Suppose that $Y$ is the point in $\P^1$ such that the following $6$-term cross ratio $$[\infty, X_3, X_1, Y, X_3, X_5] = -1.$$ Also suppose that $Y'$ is the point such that the $4$-term cross ratio $$[X_3, \frac{X_1+X_3}{2}, Y', \frac{X_3+X_5}{2}] = -1.$$ Then $Y = Y'.$ 
\end{lemma}

\begin{proof}
This follows from an easy calculation. We solve for $Y$ and $Y'$ from the corresponding cross ratios and obtain that $$Y = Y'  =\frac{X_3^2 - X_1 X_5}{2 X_3 - X_1 - X_5}.$$
\end{proof}

\begin{corollary}\label{corollary:odd_indices}
Let $A_0, A_1, ...$ be defined and labelled as above, then $T_1(A_{-1}, A_1) = (A_1, A_3)$, where $A_{-1} = \{\infty, \infty, ..., \infty\}$ consists of $n$ points at $\infty \in \R P^1$. 
\end{corollary}

\begin{proof}
 $A_0$ consists of $n$ points at $\infty$, so by the cross ratio condition we  have that $X_{(2, 2k)} = \frac{1}{2} (X_{(1,2k-1)} + X_{(1, 2k+1)})$ for all $k$. Thus $X_{(3,2k+1)}$ is determined by cross ratio $$[X_{(1,2k+1)}, \frac{X_{(1,2k-1)}+X_{(1,2k+1)}}{2}, X_{(3,2k+1)}, \frac{X_{(1,2k+1)} + X_{(1,2k+3)}}{2} ] = -1.$$  By the lemma we know that 
 $$[\infty, X_{(1,2k+1)},X_{(1,2k-1)}, X_{(3,2k+1)}, X_{(1,2k+1)}, X_{(1,2k+3)}] = -1.$$ By definition, this shows that $T_1(A_{-1}, A_1) = (A_1, A_3)$
 
\end{proof}

We restate the theorem that we want to prove:

\begin{theorem*}[\textbf{\ref{T005}}] Let $A_0$ and $A_1$ be as defined, and suppose that $$X_{(0,2)}  =  X_{(0,4)} = ... = X_{(0,2n)}  = \infty \in \R P^1,$$ also suppose that the cross ratio iterations are well defined for $2k-1$ steps, then $$ X_{(2n-1,k)} =X_{(2n, k+1)} = \frac{1}{n} \Big(X_{(1,1)} + X_{(1,3)} + ... +  X_{(1, 2n-1)}\Big) $$ for all $k =1, 3, ..., 2n-1$. 
\end{theorem*}

\begin{proof} 
First we consider the rows $A_i$ of even indices. We know that for all $i$ such that $1 \le i \le n-1$, $T_1^{i} (A_0, A_2) = (A_{2i}, A_{2i+2})$ by repetitively apply Lemma \ref{lemma:grid_embedding}. Especially, we have that $$T_1^{n-1} (A_0, A_2)= (A_{2n-2}, A_{2n}).$$ Theorem \ref{T008} tells us that $$X_{2n,2} = X_{2n,4} = ... = X_{2n, 2n} = \frac{1}{n}  \left(X_{(2,2)} + X_{(2,4)} + ... +  X_{(2, 2n)}\right).$$ 

For the odd indices, we apply Corollary \ref{corollary:odd_indices} and Lemma \ref{lemma:grid_embedding} and obtain that $$T_1^{n-1} (A_{-1}, A_1) = (A_{2n-3}, A_{2n-1}). $$ Again by Theorem \ref{T008}, we have that $$ X_{(2n-1,1)} = ... X_{(2n-1, 2n-1)}= \frac{1}{n} \Big(X_{(1,1)} + X_{(1,3)} + ... +  X_{(1, 2n-1)}\Big).$$
It is clear that $$X_{(1,1)} + X_{(1,3)} + ... +  X_{(1, 2n-1)} = X_{(2,2)} + X_{(2,4)} + ... +  X_{(2, 2n)}$$ since $A_0$ consists of infinities, hence we reach the statement of the theorem. 

\end{proof}

\section*{acknowledgement} 
This research was part of the 2014 combinatorics REU program at the University of Minnesota, Twin Cities, and was supported by RTG grant NSF/DMS-1148634. I want to thank my mentor, Max Glick, for sharing his conjectures and providing insightful guidance. I thank the TA Kevin Dilks for  his help, especially in editing the paper. I would also like to thank the math department of the university for its hospitality and Gregg Musiker for organizing the program. In addition, I want to thank Richard Schwartz for his helpful comments and encouragements, and  Ka Yu Tam and Victor Reiner for interesting conversations.  \\


\end{document}